\documentclass{article}
\usepackage[vcentering]{geometry}

\usepackage{graphicx,xcolor}
\usepackage{tikz}

\usepackage{makeidx}
\makeindex

\usepackage[backend=biber,bibencoding=utf8,giveninits=true]{biblatex}
\renewcommand{\labelnamepunct}{\addcolon\addspace}

\DeclareFieldFormat[article]{title}{\mkbibemph{#1}} %article要素のtitleを斜体に
\DeclareFieldFormat[inbook]{title}{\mkbibemph{#1}} %inbook要素のtitleを斜体に
\DeclareFieldFormat[inproceedings]{title}{\mkbibemph{#1}} %inbook要素のtitleを斜体に
\DeclareFieldFormat[thesis]{title}{\mkbibemph{#1}} %inbook要素のtitleを斜体に
\DeclareFieldFormat[article]{journaltitle}{#1}
\DeclareFieldFormat[inbook]{booktitle}{#1}
\DeclareFieldFormat[inproceedings]{booktitle}{#1}
\DeclareFieldFormat[article]{year}{(#1)}
\DeclareFieldFormat[misc]{year}{(#1)}
\DeclareFieldFormat[article]{pages}{#1}
\DeclareFieldFormat[inbook]{pages}{#1}
\DeclareFieldFormat[inproceedings]{pages}{#1}
\DeclareFieldFormat[book]{title}{#1}
\DeclareFieldFormat*{volume}{\mkbibbold{#1}}
\renewbibmacro{in:}{In}  %in: journalnameのin:をとる
\DeclareBibliographyDriver{book}{%
	\printnames{author}%
	\labelnamepunct\newblock
	\printfield{title}%
	\subtitlepunct
	\printfield{subtitle}
	\newunit
	\printlist{publisher}%
	\newunit
	\printlist{address}
	\newunit
	\printfield{year}%
	\finentry}
\DeclareBibliographyDriver{article}{%
	\printnames{author}%
	\labelnamepunct\newblock
	\printfield{title}%
	\newunit\newblock
	\printfield{journaltitle}%
	\addspace
	\printfield{volume}%
	\addspace%
	\printfield{year}%
	\newunit
	\printfield{pages}
	\finentry}
\DeclareBibliographyDriver{inbook}{%
	\printnames{author}%
	\labelnamepunct\newblock
	\printfield{title},%
	\addspace
	\usebibmacro{in:}
	\printfield{booktitle}%
	\newunit
	\printlist{publisher}
	\newunit
	\printfield{year}%
	\newunit
	\printfield{pages}
	\finentry}
\DeclareBibliographyDriver{inproceedings}{%
	\printnames{author}%
	\labelnamepunct\newblock
	\printfield{title},%
	\addspace
	\usebibmacro{in:}
	\printfield{booktitle}%
	\newunit
	\printlist{publisher}
	\newunit
	\addspace%
	\printfield{year}%
	\newunit
	\printfield{pages}
	\finentry}
\DeclareBibliographyDriver{misc}{%
	\printnames{author}%
	\labelnamepunct\newblock
	\printfield{title}%
	\newunit\newblock
	\printfield{note}%
	\addspace
	\printfield{year}%
	\finentry}
\DeclareBibliographyDriver{thesis}{%
	\printnames{author}%
	\labelnamepunct\newblock
	\printfield{title}%
	\newunit\newblock
	\printlist{institution}%
	\newunit
	\printfield{year}%
	\addspace
	\printfield{note}%
	\finentry}

\usepackage{amsmath,amssymb}
\usepackage{rsfso}
\usepackage{amscd} %xypicを使うため
\usepackage[all]{xy} %可換図式を書く
\usepackage[amsmath,thmmarks]{ntheorem}
\theoremstyle{plain}
\theoremseparator{.}
\newtheorem{Th}{Theorem}[section]
\newtheorem{Prop}[Th]{Proposition}
\newtheorem{Lem}[Th]{Lemma}
\newtheorem{Cor}[Th]{Corollary}
\theorembodyfont{\normalfont}
\newtheorem{Def}[Th]{Definition}

\newtheorem{Rem}[Th]{Remark}
\theoremstyle{nonumberplain}
\newtheorem{ackn}{Acknowledgement}
\theoremheaderfont{\itshape}
\theoremsymbol{\ensuremath{\square}}
\newtheorem{proof}{Proof}
\newcommand{\cadlag}{c\`{a}dl\`{a}g}

\newcommand{\lBrack}{[\![}
\newcommand{\rBrack}{]\!]} 

\allowdisplaybreaks[1]

\renewcommand{\thefootnote}{\fnsymbol{footnote}}

\numberwithin{equation}{section}

\usepackage[hyperfootnotes=false]{hyperref}
\usepackage[scale=1]{tgtermes}

\begin{document}

\title{Remarks on F{\"o}llmer's pathwise It{\^o} calculus}
\author{Yuki Hirai}
\date{}
\maketitle

\begin{abstract}
We extend some results about F{\"o}llmer's pathwise It{\^o} calculus
that have only been derived for continuous paths to {\cadlag} paths with quadratic variation.
We study some fundamental properties of pathwise It{\^o} integrals with respect to
{\cadlag} integrators, especially associativity and the integration by parts formula.
Moreover, we study integral equations with respect to pathwise It{\^o} integrals.
We prove that some classes of integral equations, which can be explicitly solved in the
usual stochastic calculus, can also be solved within the framework of F{\"o}llmer's calculus.
\end{abstract}

\renewcommand{\thefootnote}{\arabic{footnote}}

%\tableofcontents
\section{Introduction}

F{\"o}llmer's pathwise It{\^o} calculus originated in a seminal paper that introduced the notion of the quadratic variation of
a deterministic {\cadlag} path and proved a non-probabilistic version of It{\^o}'s formula~\cite{Foellmer_1981}.
In recent years, some related works have appeared, addressing applications to mathematical finance.
Most of these works only deal with continuous paths of quadratic variation while the results of the original paper include discontinuous cases.
The main purpose of the present paper is to generalize some of these results so that they can be applied to {\cadlag} paths with quadratic variation.

There are several approaches to pathwise constructions of stochastic integrations.
We can refer to
%%stochastic integralのpathwise approximation
%Bichteler~\cite{Bichteler_1981}, 
%Karandikar~\cite{Karandikar_1995}, 
%Willinger and Taqqu~\cite{Willinger_Taqqu_1988,Willinger_Taqqu_1989},
%%quasi-sure
%Nutz~\cite{Nutz_2012},
%%typical price paths
%Perkowski and Pr{\"o}mel~\cite{Perkowski_Proemel_2016},
%{\L}ochowski~\cite{Lochowski_2015},
%{\L}ochowski, Perkowski, and Pr{\"o}mel~\cite{Lochowski_Perkowski_Proemel_2016}
%Vovk~\cite{Vovk_2016},
%%Russo-Vallois symmetric integration
%Russo and Vallois~\cite{Russo_Vallois_2007},
%%Ito functional calculus, non-anticipating functional
%Dupire~\cite{Dupire_2009},
%Cont and Fourni{\'e}~\cite{Cont_Fournie_2010,Cont_Fournie_2010b,Cont_Fournie_2013}, %non-anticipative functional
%Ananova and Cont~\cite{Ananova_Cont_2017},
%%何だっけ？
%%Bender, Sottinen, and Valkeila~\cite{Bender_Sottinen_Valkeila_2008} %やっぱ削除
%stochastic integralのpathwise approximation
\cite{Bichteler_1981,Karandikar_1995,Willinger_Taqqu_1988,Willinger_Taqqu_1989,
Nutz_2012,
Perkowski_Proemel_2016,Lochowski_2015,Lochowski_Perkowski_Proemel_2016,
Vovk_2016,
Russo_Vallois_2007,
Dupire_2009,Cont_Fournie_2010,Cont_Fournie_2010b,Cont_Fournie_2013,Ananova_Cont_2017},
%何だっけ？
%Bender, Sottinen, and Valkeila~\cite{Bender_Sottinen_Valkeila_2008} %やっぱ削除
for example.
Among them, we think that F{\"o}llmer's approach is intuitively clear (especially from a financial application viewpoint) and needs only elementary arguments, which should be regarded as an advantage.
The rough path theory,
pioneered by
%Lyons~
\cite{Lyons_1998}
(see, also, 
%Gubinelli~
\cite{Gubinelli_2004}
and
%Friz and Shekhar~
\cite{Friz_Shekhar_2017}), 
should be considered as another important approach.
Some authors have studied the relation between F{\"o}llmer's pathwise It{\^o} calculus and rough path theory: see
%Perkowski and Pr{\"o}mel~
\cite{Perkowski_Proemel_2016}
and
%Friz and Hairer~
\cite{Friz_Hairer_2014}, for example.

Let us give an outline of F{\"o}llmer's pathwise It{\^o} calculus.
Let $\Pi = (\pi_n)_{n \in \mathbb{N}}$ be a sequence of partitions of $\mathbb{R}_{\geq 0}$ such that $\lvert \pi_n \rvert \to 0$ as $n \to \infty$.
We say that a {\cadlag} path $X \colon \mathbb{R}_{\geq 0} \to \mathbb{R}$ has quadratic variation along $\Pi$
if there exists a {\cadlag} increasing function $[X,X]$ such that for all $t \in \mathbb{R}_{\geq 0}$
\begin{enumerate}
\item $\sum_{t_i \in \pi_n} (X_{t_{i+1} \wedge t - X_{t_i \wedge t}})^2$ converges to $[X,X]_t$ as $n \to \infty$,
\item $\Delta [X,X]_{t} = (\Delta X_t)^2$.
\end{enumerate}
An $\mathbb{R}^d$-valued {\cadlag} path $X = (X^1,\dots,X^d)$ has quadratic variation if, for each $i$ and $j$, the real-valued path $X^i + X^j$ has quadratic variation.
%F{\"o}llmer~
In \cite{Foellmer_1981}, it is proved that if $X$ has quadratic variation, then for any $f \in C^2(\mathbb{R}^d)$
the path $t \mapsto f(X_t)$ satisfies It{\^o}'s formula.
That is, in $1$-dimensional case,
\begin{align*}
f(X_t) - f(X_0)
& = \int_0^t f'(X_{s-}) dX_s + \frac{1}{2} \int_0^t f''(X_{s-}) d[X,X]_s   \notag \\
& \quad + \sum_{0 < s \leq t} \left\{ \Delta f(X_s) - f'(X_{s-}) \Delta X_s - \frac{1}{2} f''(X_{s-}) (\Delta X_s)^2 \right\}
\end{align*}
holds for all $t \in \mathbb{R}_{\geq 0}$.
Here, the first term of the right-hand side, which we call the It{\^o}-F{\"o}llmer integral, is defined as the limit of a sequence of non-anticipative Riemann sums.
F{\"o}llmer's theorem claims that the It{\^o}-F{\"o}llmer integral $\int_0^t f'(X_{s-}) dX_s$ exists and it satisfies the above formula.

The results of this paper are divided into two parts.
Our first aim is to establish several calculation rules for It{\^o}-F{\"o}llmer integrals.
We slightly extend the theorem of F{\"o}llmer to a path of the form $t \mapsto f(A_t,X_t)$
with $f \in C^{1,2}(\mathbb{R}^{m} \times \mathbb{R}^d)$ and 
an $m$-dimensional path $A$ of locally finite variation.
Thus we find that for a path of the form $t \mapsto \nabla_x f(A_t,X_t)$ the It{\^o}-F{\"o}llmer integral
$\int_0^t \langle \nabla_x f(A_{s-},X_{s-}), dX_s \rangle$ exists.
We call  a path of the form $\nabla_x f(A_t,X_t)$ an admissible integrand of $X$.
Because a {\cadlag} path defined as the It{\^o}-F{\"o}llmer integral of an admissible integrand has quadratic variation,
we can consider the It{\^o}-F{\"o}llmer integral by a path of this type.
For this case, we prove that It{\^o}-F{\"o}llmer integral satisfies the so-called associativity rule (Theorem~\ref{2.3m}).
This was already proved by
%Schied~
\cite{Schied_2014} for continuous integrators;
we extend it to general {\cadlag} paths with quadratic variation.
An integration by parts formula is then proved as an application of associativity and F{\"o}llmer's theorem (Corollary~\ref{2.3s}).

The next aim is to study integral equations with respect to It{\^o}-F{\"o}llmer integrals.
By using the calculation rules mentioned above, we can solve certain classes of integral equations.
First, we solve linear integral equations:
the solutions to the homogeneous linear equations are given by the Doleans-Dade exponentials
(Proposition~\ref{3.1c}).
The solutions to the inhomogeneous linear equations are constructed by using the result for the homogeneous case (Proposition~\ref{3.1h}),
which is a re-interpretation of the result by
%Jaschke~
\cite{Jaschke_2003} in our pathwise setting.
These results are applied to compute a portfolio insurance strategy in finance (Proposition~\ref{6.1l}).
Further, we solve a path-dependent equation, which is called the drawdown equation, introduced by
%Cararro, El Karoui, and Ob{\l}{\'o}j~
\cite{Carraro_ElKaroui_Obloj_2012}: their results are re-interpreted in our pathwise setting (Proposition~\ref{3.3i}).

%outline
Let us describe the structure of this paper.
In Section 2, we study the basic properties of It{\^o}-F{\"o}llmer integrals.
In Section 2.1 we define a quadratic variation and show some propositions that will be used in following sections.
It{\^o}-F{\"o}llmer integrals are introduced and the associated It{\^o} formula
for It{\^o}-F{\"o}llmer integrals is proved in Section 2.2. 
The associativity of It{\^o}-F{\"o}llmer integrals and an integration by parts formula are proved in Section 2.3.
Pathwise quadratic variations of semimartingales are discussed in Section 2.4.
In Section 3, integral equations with respect to It{\^o}-F{\"o}llmer integrals are studied.
In Section 3.1, linear integral equations are explicitly solved.
In Section 3.2, certain nonlinear integral equations are solved.
In Section 3.3, drawdown equations, which are path-dependent equations, are studied.
Applications to financial topics satisfying certain kinds of floor constraints are considered in Section 3.4
.

Here we give some notation and terminology that are frequently used in this paper.
We write $\mathbb{N} = \{0,1,2,\dots \}$ and $\mathbb{N}_{\geq 1} = \{1,2,3,\dots\}$.
The set of real numbers is denoted by $\mathbb{R}$,
and we set $\mathbb{R}_{\geq 0} = \left[ 0,\infty \right[ = \{ r \in \mathbb{R} \mid r \geq 0 \}$.
For $x,y \in \mathbb{R}^d$, the standard Euclidean norm and inner product are denoted by $\lVert x \rVert$ and $\langle x,y \rangle$, respectively. 

A function $X \colon \mathbb{R}_{\geq 0} \to \mathbb{R}^d$ is called a {\cadlag} path if it is right continuous and have finite left-hand limits at all $t \in \mathbb{R}_{\geq 0}$. The set of all $\mathbb{R}^d$-valued {\cadlag} paths is denoted by $D(\mathbb{R}_{\geq 0},\mathbb{R}^d)$. For $X \in D(\mathbb{R}_{\geq 0},\mathbb{R})$, we write $X_t = X(t)$ and $\Delta X_t = X_t - X_{t-}$. For convenience, we define $X_{0-} = X_0$ and $\Delta X_0 = 0$.
If $X = (X^1,\dots,X^d) \in D(\mathbb{R}_{\geq 0},\mathbb{R}^d)$, we write $\Delta X_t = (\Delta X^1_t,\dots,\Delta X^d_t)$.
Let us recall that if $X$ is a {\cadlag} path satisfying $\lvert \Delta X \rvert \leq c$ on $[0,t]$,
then for any $\varepsilon > 0$ there exists a $\delta > 0$ such that $\lvert s - u \rvert < \delta$ and $s,u \in [0,t]$ implies $\lvert X_s - X_u \rvert < \varepsilon + c$.

The symbol $FV_{\mathrm{loc}}$ denotes the set of all real-valued {\cadlag} paths of locally finite variation.
The total variation of $A \in FV_{\mathrm{loc}}$ on $[0,t]$ is denoted by $V(A)_t$.
If $A \in FV_{\mathrm{loc}}$, the series $\sum_{0 \leq s \leq t} \Delta A_s$ converges absolutely for all $t \in \mathbb{R}_{\geq 0}$. Then $A^{\mathrm{d}}_t := \sum_{0 <s \leq t} \Delta A_s$
and $A^{\mathrm{c}} := A - A^{\mathrm{d}}$
are called the purely discontinuous part and the continuous part of $A$, respectively.
For $A \in FV_{\mathrm{loc}}$, we write $\int_{0}^{t} f(s) dA_s = \int_{]0,t]} f(s) dA_s$
when the Lebesgue-Stieltjes integral in the right hand is well-defined.

In this paper, we suppose that any partition $\pi = (t_i)_{i \in \mathbb{N}}$ of $\mathbb{R}_{\geq 0}$ satisfies
$0 = t_0 < t_1 < t_2 < \dotsb \to \infty. $
For a partition $\pi = (t_i)_{i \in \mathbb{N}}$, we define $\lvert \pi \rvert = \sup_{i} \lvert t_{i+1} - t_i \rvert$.
We often identify the partition $\pi = (t_i)_{i \in \mathbb{N}}$ with the set $\{ t_1,t_2,\dots \}$, 
and use set notation for partitions.
For example, we write $t_i \in \pi$ or $\pi \subset \pi'$ for two partitions $\pi$ and $\pi'$.
Given a sequence of partitions $(\pi_n)_{n \in \mathbb{N}} = ((t^n_i)_{i \in \mathbb{N}})_{n \in \mathbb{N}}$ and an $\mathbb{R}^d$-valued {\cadlag} path $X$, 
we write $\delta^{n}_i X = X_{t^n_{i+1}} - X_{t^n_i}$ for convenience.

Given $E \subset \mathbb{R}^n$, a normed space $V$ , and a function $f:E \to V$, we define
\begin{equation*}
\omega(f;\varepsilon) = \sup \{ \lVert f(y) - f(x) \rVert_{V} \mid x,y \in E,~\lVert x-y \rVert_{\mathbb{R}^n} < \varepsilon \}.
\end{equation*}
If $f$ is uniformly continuous, we have $\omega(f;\varepsilon) \to 0$ as $\varepsilon \to 0$.

Recall that a function $f \colon U \to \mathbb{R}^n$ defined on an open subset $U \subset \mathbb{R}^m \times \mathbb{R}^d$ is of $C^{k,l}$-class ($k,l \in \mathbb{N}$) if
\begin{enumerate}
\item for each $y$, the map $x \mapsto f(x,y)$ is $k$-times continuously differentiable and all derivatives with respect to $x$ are continuous in $(x,y) \in U$,
\item for each $x$, the map $y \mapsto f(x,y)$ is $l$-times continuously differentiable and all derivatives with respect to $y$ are continuous in $(x,y) \in U$.
\end{enumerate}
The symbol $C^{k,l}(U)$ stands for the space of real-valued $C^{k,l}$-class function on $U$.
The first and second order derivatives of $f$ with respect to a variable $x$
are denoted by $\nabla_x f$ and $\nabla_x^2 f$, respectively, if they exist.
We simply write $\nabla f$ for the first order derivative with respect to all the variables.

\section{Quadratic variations and It{\^o}-F{\"o}llmer integration}

\subsection{Definition of quadratic variation and its properties}

In this subsection, we define the quadratic variation of a {\cadlag} path along a sequence of partitions,
and we study some fundamental properties of it.

The discrete quadratic variation of a real-valued {\cadlag} path $X$ along a partition $\pi$ is
defined by 
\begin{equation*}
[X,X]^{\pi}_t = \sum_{t_i \in \pi} (X_{t_{i+1} \wedge t} - X_{t_i \wedge t})^2.
\end{equation*}

\begin{Def} \label{2.1a}
Let $\Pi = (\pi_n)_{n \in \mathbb{N}}$ be a sequence of partitions of $\mathbb{R}_{\geq 0}$ such that $\lvert \pi_n \rvert \to 0$, and
let $X: \mathbb{R}_{\geq 0} \to \mathbb{R}$ be a {\cadlag} path.
We say that \emph{$X$ has quadratic variation along $\Pi$} if it satisfies the following conditions.
\begin{enumerate}
\item The sequence of {\cadlag} paths $([X,X]^{\pi_n})_{n \in \mathbb{N}}$ converges pointwise to some {\cadlag} increasing function $[X,X]^{\Pi}$.
\item $\Delta [X,X]^{\Pi}_t = (\Delta X_t)^2$ holds for all $t \in \mathbb{R}_{\geq 0}$.
\end{enumerate}
The increasing function $[X,X]^{\Pi}$ is called the \emph{quadratic variation} of $X$ along $\Pi = (\pi_n)$.
We often omit the symbol $\Pi$ and simply write $[X,X]$ if there is no ambiguity.
\end{Def}

Note that if $X$ has quadratic variation along $\Pi$, it satisfies $\sum_{0 < s \leq t}(X_s)^2 < \infty$ for all $t \in \mathbb{R}_{\geq 0}$.
For continuous paths, existence of the quadratic variation is clearly equivalent to the pointwise convergence of
$([X,X]^{\pi_n})_{n \in \mathbb{R}}$ to a continuous increasing function $[X,X]$.
This condition is, in fact, equivalent to uniform convergence on compacts:
see
%Davis, Ob{\l}{\'{o}}j, and Siorpaes~
\cite{Davis_Obloj_Siorpaes_2017}.

A typical example of quadratic variation is the quadratic variation of a semimartingale. 
We discuss this case in Section 3.1.
Another example is a Dirichlet process: 
see
%F{\"o}llmer~
\cite{Foellmer_1981,Foellmer_1981b}.
Also, paths of certain Gaussian processes have quadratic variation:
see, for example,
%Baxter~
\cite{Baxter_1956}.
The existence of quadratic variations under non-probabilistic settings is studied in 
%Vovk~
\cite{Vovk_2012,Vovk_2015}.

\begin{Rem} \label{2.1aa0a}
In general, the dependence of the quadratic variation $[X,X]^{\Pi}$
on the choice of a sequence of partitions is inevitable.
This problem is discussed in
%Davis, Ob{\l}{\'o}j, and Siorpaes~
\cite[Section 7]{Davis_Obloj_Siorpaes_2017}.
\end{Rem}

\begin{Def} \label{2.1b}
A $d$-dimensional {\cadlag} path $X = (X^1,\dots,X^d)$ has quadratic variation along $\Pi$
if, for each $i,j \in \{ 1,\dots,d\}$, $X^i + X^j$ has quadratic variation along $\Pi$.
\end{Def}

\begin{Rem} \label{2.1c}
In general, the existence of the quadratic variations of $X$ and $Y$ does not implies
that of $X+Y$: see
%Schied~
\cite{Schied_2016}.
\end{Rem}

The symbol $QV(\Pi;\mathbb{R}^d)$
denotes the set of all $\mathbb{R}^d$-valued {\cadlag} paths that have quadratic variation along $\Pi$.
We also write $QV(\Pi)$, or $QV$ if $d = 1$.
For $(X,Y) \in QV(\Pi,\mathbb{R}^2)$, we define
\begin{equation*}
[X,Y]^{\Pi} = \frac{1}{2} \left( [X+Y,X+Y]^{\Pi} - [X,X]^{\Pi} - [Y,Y]^{\Pi} \right).
\end{equation*}
The path $[X,Y]^{\Pi}$ is of locally finite variation, by definition.
$[X,Y]^{\Pi}$ is called the quadratic covariation of $X$ and $Y$ along $\Pi$.
For $X,Y \in QV(\Pi;\mathbb{R})$, the condition $(X,Y) \in QV(\Pi;\mathbb{R}^2)$ is
clearly equivalent to the following two conditions.
\begin{enumerate}
\item The function $t \mapsto \sum_{t_i \in \pi_n}(X_{t_{i+1} \wedge t} - X_{t_i \wedge t})(Y_{t_{i+1} \wedge t}-Y_{t_i \wedge t})$ converges pointwise to some {\cadlag} path $[X,Y]^{\Pi}$ of locally finite variation.
\item $\Delta [X,Y]^{\Pi}_t = \Delta X_t \Delta Y_t$ holds for all $t \in \mathbb{R}_{\geq 0}$.
\end{enumerate}

Let $\delta_a$ be the Dirac measure on $\mathbb{R}_{\geq 0}$ concentrated at $a$,
let $X$ and $Y$ be {\cadlag} paths, and let $\pi$ be a partition of $\mathbb{R}_{\geq 0}$.
We define locally finite measures $\mu_X^{\pi}$ and $\mu_{X,Y}^{\pi}$ by 
\begin{equation*}
\mu_{X,Y}^{\pi} = \sum_{t_i \in \pi}(X_{t_{i+1}} - X_{t_i}) (Y_{t_{i+1}} - Y_{t_i}) \delta_{t_i}
\end{equation*}
and $\mu^{\pi}_X = \mu^{\pi}_{X,X}$. Then,
\begin{equation*}
\mu_{X,Y}^{\pi} ([0,t]) = \sum_{t_i \in \pi \cap [0,t]} (X_{t_{i+1}} - X_{t_i}) (Y_{t_{i+1}} - Y_{t_i}).
\end{equation*}

\begin{Prop} \label{2.1d}
Let $X$ and $Y$ be real-valued {\cadlag} paths and let $\Pi = (\pi_n)_{n \in \mathbb{N}}$ be a sequence of partitions such that $\lvert \pi_n \rvert \to 0$.
Then, for all $t \in \mathbb{R}_{\geq 0}$, the following two conditions are equivalent.
\begin{enumerate}
\item The sequence $([X,Y]^{\pi_n}_t)_{n \in \mathbb{N}}$ converges.
\item The sequence $(\mu_{X,Y}^{\pi_n}([0,t]))_{n \in \mathbb{N}}$ converges.
\end{enumerate}
If these conditions are satisfied, both sequences have the same limit.
\end{Prop}

\begin{proof}
It suffices to give a proof when $X=Y$.
If $t \in [t_i,t_{i+1}[$ and $t_i \in \pi_n$, we have
\begin{equation} \label{2.1e}
\left\lvert [X,X]^{\pi_n}_t - \mu_X^{\pi_n}([0,t]) \right\rvert
= \left\lvert (X_{t_{i+1}} - X_{t_i})^2 - (X_{t} - X_{t_i})^2 \right\rvert \leq 4 \sup_{s \in [0,t+\lvert \pi_n \rvert] } \lvert X_{s} \rvert \left\lvert X_{t_{i+1}} - X_{t} \right\rvert.
\end{equation}
By the assumption $\lvert \pi_n \rvert \to 0$ and right continuity of $X$, each side of \eqref{2.1e} converges to $0$.
This implies the assertion.
\end{proof}

We see that, by Proposition~\ref{2.1d}, 
there is essentially no difference between the two definitions of quadratic variation, 
convergence of $\mu_X^{\pi_n}$ and convergence of $[X,X]^{\pi_n}$.
Proposition~\ref{2.1d} yields that, for $X \in QV(\Pi)$,
the sequence of distribution functions associated with $(\mu_X^{\pi_n})_n$ converges pointwise to $[X,X]$.
Therefore, the sequence of measures $(\mu_X^{\pi_n})$ converges vaguely to the Stieltjes measure generated by $[X,X]$, denoted by $\mu_X^{\Pi}$.
This implies, in the continuous path case, that
\begin{equation} \label{2.1f}
\lim_{n \to \infty} \int_{[0,t]} f(X_s) \mu^{\pi_n}_X(ds) = \int_{[0,t]} f(X_{s}) d[X,X]_s
\end{equation}
holds for every bounded continuous function $f$.
%This is not obvious if $X$ is not continuous.
The following lemma claims that a similar property holds for a general {\cadlag} path of $QV(\Pi)$.
This is proved as a part of the proof of the It{\^o} formula in
%F{\"o}llmer~
\cite{Foellmer_1981}.
See also
%Medvegyev~
\cite[Theorem 6.52]{Medvegyev_2007} for a proof.
Here we will give a slightly extended version of this result.

\begin{Lem} \label{2.1g}
Assume that $X = (X^1,\dots,X^d) \in QV(\Pi;\mathbb{R}^d)$ and $Y = (Y^1,\dots,Y^m) \in D(\mathbb{R}_{\geq 0},\mathbb{R}^m)$.
Then for any continuous function $g:\mathbb{R}^d \times \mathbb{R}^{m} \to \mathbb{R}$, 
$t \in \mathbb{R}_{\geq 0}$, and $i,j \in \{1,\dots,d\}$, we have
\begin{equation*}
\lim_{n \to \infty} \int_{[0,t]} g(X_s,Y_s) \mu_{X^i,X^j}^{\pi_n}(ds) = \int_{[0,t]} g(X_{s-},Y_{s-}) d[X^i,X^j]^{\Pi}_s.
\end{equation*}
\end{Lem}

\begin{proof}
If $d=1$, the assertion follows directly from Lemma~\ref{a1b}.
For the general case, we apply Lemma~\ref{a1b} to each $X^i$ and $X^i + X^j$, 
and combine their convergence and the definition of quadratic covariation.
\end{proof}

In semimartingale theory,
a process defined as the composition of a semimartingale and a $C^1$-class function
again has quadratic variation 
(%Meyer~
\cite[Chapitre VI. 5. Theorem]{Meyer_1976}).
%F{\"o}llmer~
\cite{Foellmer_1981} mentions that a similar result holds within this framework.
Here we prove a slightly extended version of that result.

\begin{Prop} \label{2.1m}
Let $X \in QV(\Pi,\mathbb{R}^d)$, let $A \in FV_{\mathrm{loc}}^m$,
and let $f \in C^{0,1}(\mathbb{R}^m \times \mathbb{R}^d)$ be  locally Lipschitz.
Then $f(A,X)$ has the quadratic variation along $\Pi$ given by
\begin{equation} \label{2.1n}
\left\lbrack f(A,X),f(A,X) \right\rbrack_t = \sum_{k,l=1}^d \int_{0}^{t} \left( \frac{\partial f}{\partial x_k}\frac{\partial f}{\partial x_l} \right)(A_{s-},X_{s-}) d[X^k,X^l]^{\mathrm{c}}_s + \sum_{0 < s \leq t} (\Delta f(A_s,X_s))^2.
\end{equation}
\end{Prop}

\begin{proof}
Fix $t>0$. Because the image of $[0,t]$ under $(A,X)$ is bounded in $\mathbb{R}^{m+d}$,
we can assume, without loss of generality, that $f$ has compact support .
First note that there is a positive constant $C$ such that
\begin{equation*}
\sum_{0 < s \leq t} \left(\Delta f(A_{s},X_{s}) \right)^2 \leq C \left( \sum_{0 < s \leq t} (\Delta X_s)^2 + \sum_{0 < s \leq t} (\Delta A_s)^2 \right) < \infty
\end{equation*}
thanks to the Lipschitz continuity of $f$.

Let 
\begin{gather*}
D(t) = \{ s \in [0,t] \mid \Delta (A,X)_s \neq 0 \},     \\
D_p(t) =\left\{ s \in D(t) \,\middle\vert\, \lVert \Delta (A,X)_s \rVert_{\mathbb{R}^{m+d}} \geq \frac{1}{p}  \right\}, \quad p \in \mathbb{N}_{\geq 1}.
\end{gather*}
Then each $D_p(t)$ is a finite set and $\bigcup_{p \geq 1} D_p(t) = D(t)$.
By convention, we use the following notation.
\begin{equation*}
\sum_{(1,n,p)} = \sum_{\substack{ t_i \in \pi_n \cap [0,t] \\ ]t_i,t_{i+1}] \cap D_p(t) \neq \emptyset}}, \qquad \sum_{(2,n,p)} = \sum_{t_i \in \pi_n \cap [0,t]} - \sum_{(1,n,p)}.
\end{equation*}
Using this notation and Taylor's theorem, we have
\begin{align} \label{2.1q}
& \sum_{t_i \in \pi_n \cap [0,t]} \left\{ \delta^n_i f(A,X) \right\}^2  \notag \\
& = \sum_{(1,n,p)}  \left\{ \delta^n_i  f(A,X) \right\}^2 + \sum_{(2,n,p)} \left\{ f(A_{t^n_{i+1}},X_{t^n_{i+1}})- f(A_{t^n_i},X_{t^n_{i+1}}) \right\}^2  \notag \\
& \quad + 2\sum_{(2,n,p)}  \left\{ f(A_{t^n_{i+1}},X_{t^n_{i+1}})- f(A_{t^n_i},X_{t^n_{i+1}}) \right\} \left\{ f(A_{t^n_i},X_{t^n_{i+1}})- f(A_{t^n_i},X_{t^n_i}) \right\}  \notag \\
& \quad + \sum_{t_i \in \pi_n \cap [0,t]} \left\langle \nabla_x f(A_{t^n_i},X_{t^n_i}), \delta^n_i X \right\rangle^2 - \sum_{(1,n,p)} \left\langle \nabla_x f(A_{t^n_i},X_{t^n_i}), \delta^n_i X \right\rangle^2  \notag \\
& \quad + 2 \sum_{(2,n,p)} \left\langle r^n_i, \delta^n_i X \right\rangle \left\langle \nabla_x f(A_{t_i},X_{t_i}), \delta^n_i X \right\rangle + \sum_{(2,n,p)} \left\langle r^n_i, \delta^n_i X \right\rangle^2    \notag \\
& =: I_1^{(n)} + I_2^{(n)} + 2 I_3^{(n)} + I_4^{(n)} - I_5^{(n)} + 2 I_6^{(n)} + I_7^{(n)},
\end{align}
where 
\begin{equation*}
r^n_i = \int_{[0,1]} \nabla_x f(A_{t^n_i},X_{t^n_i}+s \,\delta^n_i X) ds - \nabla_x f(A_{t^n_i},X_{t^nt_i}) \in \mathbb{R}^d.
\end{equation*}
We now consider the behavior of each term of the right-hand side of \eqref{2.1q}.

Since $D_1$ is finite, it is easy to verify that
\begin{align*}
\lim_{n \to \infty} I_1^{(n)} & = \sum_{s \in D_p(t)} (\Delta f(A_s,X_s))^2,     \\
\lim_{n \to \infty} I_5^{(n)} & = \sum_{s \in D_p(t)} \sum_{k,l=1}^d \left( \frac{\partial f}{\partial x_k}\frac{\partial f}{\partial x_l} \right) (A_{s-},X_{s-}) \Delta X^k_s \Delta X^l_s.
\end{align*}
Lipschitz continuity of $f$ implies that
\begin{equation*} 
\varlimsup_{n \to \infty} I_2^{(n)} \leq  \frac{K}{p} \sum_{k=1}^{m} V(A^k)_t, \quad
\varlimsup_{n \to \infty} \left\lvert I_3^{(n)} \right\rvert \leq \frac{K}{p} \, \sum_{k=1}^m V(A^k)_t.
\end{equation*}
holds for a positive constant $K$.
The convergence
\begin{align*}
\lim_{n \to \infty} I_4^{(n)} = \sum_{k,l=1}^d \int_0^t \left( \frac{\partial f}{\partial x_k} \frac{\partial f}{\partial x_l} \right) (A_{s-},X_{s-}) d[X^k,X^l]_s
\end{align*}
follows from Lemma~\ref{2.1g}.
Moreover, we have
\begin{equation*}
\varlimsup_{n \to \infty} \left\lvert I_6^{(n)} \right\rvert \leq \omega \left( \nabla_x f;\frac{2}{p} \right) K' \sum_{k=1}^d [X^k,X^k]_t,  \quad 
\varlimsup_{n \to \infty} I_7^{(n)} \leq \omega \left( \nabla_x f;\frac{2}{p} \right)^2 \sum_{k=1}^d [X^k,X^k]_t,
\end{equation*}
where $K' = \sup_{s \in [0,t]} \lVert \nabla_x f(A_s,X_s) \rVert$.

From \eqref{2.1q} and the above, we see that
\begin{align} \label{2.1y}
& \varlimsup_{n \to \infty} \left\lvert \sum_{t_i \in \pi_n \cap [0,t]} \big\lbrack \delta^n_i f(A,X) \big\rbrack^2 - \left( \text{RHS of \eqref{2.1n}} \right) \right\rvert   \notag \\
& \leq  C\left( \frac{1}{p} + \omega \left( \nabla_x f; \frac{2}{p} \right) + \omega \left( \nabla_x f; \frac{2}{p} \right)^2 \right)  + \left\lvert \sum_{s \in D(t) \setminus D_p(t)} \left( \Delta f(A_s,X_s) \right)^2 \right\rvert  \notag \\
& \quad + \left\rvert \sum_{s \in D(t) \setminus D_p(t)} \sum_{k,l=1}^d \left( \frac{\partial f}{\partial x_k}\frac{\partial f}{\partial x_l} \right) (A_{s-},X_{s-}) \Delta X^k_s \Delta X^l_s \right\rvert
\end{align}
holds for some positive constant $C$.
Let $p \to \infty$, and then every term of the right-hand side of \eqref{2.1y} converges to $0$.
This completes the proof.
\end{proof}

\begin{Rem} \label{2.1ya}
If the functions $f$ of Proposition~\ref{2.1m} are defined on only a subset $U \subset \mathbb{R}^{m+d}$,
the assertion is still valid provided that $(A,X)$ satisfies the condition that, for each $t$,
the image $(A,X)([0,t])$ is relatively compact in $U$.
\end{Rem}

The following results are direct consequences of Proposition~\ref{2.1m}. 

\begin{Cor} \label{2.1z}
\begin{enumerate}
\item If $A \in FV_{\mathrm{loc}}$, then $A \in QV(\Pi;\mathbb{R})$ holds
for any partitions $\Pi = (\pi_n)$ with the condition $\lvert \pi_n \rvert \to 0$,
and its quadratic variation is
\begin{equation*}
[A,A] = \sum_{0 < s \leq \cdot} (\Delta A_s)^2.
\end{equation*}
\item If $A \in FV_{\mathrm{loc}}$ and $X \in QV(\Pi;\mathbb{R})$,
we have $(X,A) \in QV(\Pi;\mathbb{R}^2)$ and 
\begin{gather*}
[X+A,X+A] = [X,X] + [A,A] + 2[X,A],  \\
[X,A] = \sum_{0 < s \leq \cdot} \Delta X_s \Delta A_s,  \\
[X+A,X+A]^{\mathrm{c}} = [X,X]^{\mathrm{c}}.
\end{gather*}
\item If $A \in FV_{\mathrm{loc}}$ and $(X,Y) \in QV(\Pi;\mathbb{R}^2)$, then $(X+A,Y) \in QV(\Pi;\mathbb{R}^2)$ and its quadratic covariation is $[X+A,Y] = [X,Y] + [A,Y]$.
\end{enumerate}
\end{Cor}

\subsection{It{\^o}-F{\"o}llmer integrals and It{\^o}'s formula}

In this subsection, we introduce the idea of the ``integral'' $\int_{0}^{t} \xi_{s-} dX_s$
with respect to $X \in QV(\Pi)$
and prove the It{\^o} formula.
We first define the It{\^o}-F{\"o}llmer integral as the limit of a sequence non-anticipative Riemann sums. 

\begin{Def} \label{2.2a}
Let $\xi: \mathbb{R}_{\geq 0} \to \mathbb{R}^d$ be a {\cadlag} path and consider $X \in QV(\Pi;\mathbb{R}^d)$.
We call the limit 
\begin{equation*}
\int_0^t \langle \xi_{s-},dX_s \rangle := \lim_{n \to \infty} \sum_{t_i \in \pi_n} \left\langle \xi_{t_{i}}, X_{t_{i+1} \wedge t} - X_{t_{i} \wedge t}\right\rangle,
\end{equation*}
\emph{the It{\^o}-F{\"o}llmer integral of $\xi$ with respect to $X$ on $]0,t]$ along the sequence $\Pi$}, if it converges to a finite number.
In the $1$-dimensional case, we simply write $\int_0^t \xi_{s-} dX_s$
instead of $\int_0^t \langle \xi_{s-},dX_s \rangle$.
\end{Def}

\begin{Rem} \label{2.2aa}
It is a well-known fact that 
$\int_{0}^{t} \xi_{s-} dX_s$ exists if 
$X$ or $\xi$ has locally finite variation.
When $X$ is in $FV_{\mathrm{loc}}$, the It{\^o}-F{\"o}llmer
integral is nothing but the Stieltjes integral of $s \mapsto \xi_{s-}$ by $X$.
\end{Rem}

%F{\"o}llmer~
\cite{Foellmer_1981} considers a different summation 
\begin{equation*}
\sum_{t_i \in \pi_n \cap [0,t]} \left\langle \xi_{t_{i}}, X_{t_{i+1}} - X_{t_{i}}\right\rangle
\end{equation*}
to define a pathwise It{\^o} integral.
However, they are equivalent under the right continuity of integrators.

\begin{Prop} \label{2.2b}
For $\xi, X \in D(\mathbb{R}_{\geq 0},\mathbb{R}^d)$ and $t \in \mathbb{R}_{\geq 0}$
the following two conditions are equivalent.
\begin{enumerate}
\item $\sum_{t_i \in \pi_n} \left\langle \xi_{t_{i}}, X_{t_{i+1} \wedge t} - X_{t_{i} \wedge t}\right\rangle$ converges.
\item $\sum_{t_i \in \pi_n \cap [0,t]} \left\langle \xi_{t_{i}}, X_{t_{i+1}} - X_{t_{i}}\right\rangle$ converges .
\end{enumerate}
If the conditions above are satisfied, the limits are equal. 
\end{Prop}

\begin{proof}
For $t \in [t^n_i,t^n_{i+1}[$ ($t^n_i \in \pi_n$), we have
\begin{align*}
& \left\lvert \sum_{t^n_i \in \pi_n} \left\langle \xi_{t^n_{i}}, X_{t^n_{i+1} \wedge t} - X_{t^n_{i} \wedge t}\right\rangle - \sum_{t^n_i \in \pi_n \cap [0,t]} \left\langle \xi_{t^n_{i}}, X_{t^n_{i+1}} - X_{t^n_{i}}\right\rangle \right\rvert %= \left\lvert \langle \xi_{t^n_i}, X_{t^n_{i+1}} - X_{t} \rangle \right\rvert   \notag \\
\leq \sup_{s \in [0,t]} \lVert \xi_{s} \rVert_{\mathbb{R}^d} \left\lVert X_{t^n_{i+1}} - X_{t} \right\rVert_{\mathbb{R}^d}.
\end{align*}
Letting $n \to \infty$ the last term converges to $0$ by the condition $\lvert \pi_n \rvert \to 0$ and the right continuity of $X$.
\end{proof}

In this paper, we consider integrands of the form $t \mapsto g(X_t)$ for a ``nice'' function $g$.
%F{\"o}llmer~
\cite{Foellmer_1981} proves that 
the It{\^o}-F{\"ollmer} integral of $t \mapsto \nabla f(X_t)$ with respect to $X$ exists for $C^2$-class $f$ and 
it satisfies the same relation as that of It{\^o}'s formula in stochastic calculus.
Here we prove a slight extension of F{\"o}llmer's theorem.

\begin{Th} \label{2.2c}
Let $X \in  QV(\Pi; \mathbb{R}^d)$ and $A \in FV_{\mathrm{loc}}^m$.
Then, for any function $f \in C^{1,2}(\mathbb{R}^m \times \mathbb{R}^d)$ and 
$t \in \mathbb{R}_{\geq 0}$, 
the It{\^o}-F{\"o}llmer integral of $\nabla_x f(A,X)$ with respect to $X$ exists and satisfies the following formula.
\begin{align} \label{2.2d}
f(A_t,X_t) - f(A_0,X_0) 
& = \sum_{k=1}^m \int_{0}^{t} \frac{\partial f}{\partial a_k}(A_{s-},X_{s-}) d(A^k)^{\mathrm{c}}_s + \int_0^t \left\langle \nabla_x f(A_{s-},X_{s-}), dX_s \right\rangle  \notag \\
& \quad + \sum_{k,l=1}^d \frac{1}{2} \int_{0}^{t} \frac{\partial^2 f}{\partial x_k \partial x_l} (A_{s-},X_{s-}) d[X^k,X^l]^{\mathrm{c}}_s  \notag \\
& \quad + \sum_{0 < s \leq t} \left\{ \Delta f(A_s,X_s) - \sum_{k=1}^d \frac{\partial f}{\partial x_k} (A_{s-},X_{s-}) \Delta X^k_s \right\} .
\end{align}
\end{Th}

\begin{proof}
Fix $t>0$.
Because $(A,X)$ has bounded range on $[0,t]$, we can assume $f$ has compact support.
First observe that 
the formula \eqref{2.2d} is equivalent to 
\begin{align} \label{2.2da}
& f(A_t,X_t) - f(A_0,X_0) - \int_0^t \left\langle \nabla_x f(A_{s-},X_{s-}), dX_s \right\rangle   \notag \\
& = \sum_{k=1}^m \int_{0}^{t} \frac{\partial f}{\partial a_k}(A_{s-},X_{s-}) dA^k_s  + \sum_{k,l=1}^d \frac{1}{2} \int_{0}^{t} \frac{\partial^2 f}{\partial x_k \partial x_l} (A_{s-},X_{s-}) d[X^k,X^l]_s   \notag \\
& \quad - \sum_{0 < s \leq t} \sum_{k,l =1}^d \frac{\partial^2 f }{\partial x_k \partial x_l}(A_{s-},X_{s-}) \Delta X_s^k \Delta X_s^l   \notag \\
& \quad + \sum_{0 < s \leq t} \left\{ \Delta f(A_s,X_s) - \sum_{k=1}^d \frac{\partial f}{\partial x_k} (A_{s-},X_{s-}) \Delta X^k_s   - \sum_{k=1}^m \frac{\partial f}{\partial a_k}(A_{s-},X_{s-}) \Delta A^k_s \right\}.
\end{align}
This follows from the fact that every series in \eqref{2.2d} and \eqref{2.2da} converges absolutely. 
The convergence of series is proved immediately by using the assumption that $f$ is a $C^{1,2}$ function
with bounded derivatives, and that $\sum_{s \leq t} (\Delta X_s)^2$ and $\sum_{s \leq t} \lvert \Delta A_s \rvert$ converges. 
So we will consider \eqref{2.2da} in this proof.

The next step of the proof is a second-order Taylor expansion.
Note that
\begin{equation} \label{2.2db}
f(A_t,X_t) - f(A_0,X_0) = \lim_{n \to \infty} \sum_{t_i \in \pi_n \cap [0,t]} \delta^n_i f(A,X)
\end{equation}
follows from the right continuity of $X$ and the continuity of $f$.
By Taylor's theorem, we have
\begin{align*}
\delta^n_i f(A,X)
& = \left\langle \nabla_a f(A_{t^n_i},X_{t^n_i}) , \delta^n_i A \right\rangle + \left\langle r^n_i, \delta^n_i A \right\rangle + \left\langle \nabla_x f(A_{t_i},X_{t_i}), \delta^n_i X \right\rangle,  \notag \\
& \quad + \frac{1}{2} \left\langle \nabla_x^2 f(A_{t^n_i},X_{t^n_i}) \delta^n_i X, \delta^n_i X \right\rangle + \left\langle R^n_i \delta^n_i X, \delta^n_i X \right\rangle.
\end{align*}
Here, $R^n_i$ and $r^n_i$ are defined by
\begin{gather*}
r^n_i = \nabla_a f(A_{t_i} + \theta_i^n (\delta^n_i A),X_{t_i+1}) - \nabla_a f(A_{t_i},X_{t_i}) \in \mathbb{R}^m,   \\
R^n_i = \nabla_x^2f(A_{t_i},X_{t_i} + \Theta_i^n (\delta^n_i X)) - \nabla_x^2 f(A_{t_i},X_{t_i}) \in M_d(\mathbb{R}),
\end{gather*}
for some $\theta^n_i$ and $\Theta^n_i \in [0,1]$.
Taking the summation over all $t_i$ in $\pi_n \cap [0,t]$, we obtain
\begin{align}
& \sum_{t_i \in \pi_n \cap [0,t]} \delta^n_i f(A,X) - \sum_{t_i \in \pi_n \cap [0,t]}  \left\langle \nabla_x f(A_{t^n_i},X_{t^n_i}), \delta^n_i X \right\rangle  \notag \\
& = \sum_{(1,n,p)} \delta^n_i f(A,X) - \sum_{(1,n,p)} \left\langle \nabla_x f(A_{t^n_i},X_{t^n_i}), \delta^n_i X \right\rangle\notag   \\
& \quad  - \sum_{(1,n,p)} \left\langle \nabla_a f(A_{t^n_i},X_{t^n_i}) ,  \delta^n_i A \right\rangle - \frac{1}{2} \sum_{(1,n,p)} \left\langle \nabla_x^2 f(A_{t^n_i},X_{t^n_i})  \delta^n_i X, \delta^n_i X \right\rangle  \notag \\
& \quad + \sum_{t_i \in \pi_n \cap [0,t]} \left\langle \nabla_a f(A_{t^n_i},X_{t^n_i}) , \delta^n_i A \right\rangle + \frac{1}{2}  \sum_{t_i \in \pi_n \cap [0,t]} \left\langle \nabla_x^2 f(A_{t_i},X_{t_i}) \delta^n_i X, \delta^n_i X \right\rangle \notag \\
& \quad + \sum_{(2,n,p)}\left\langle R^n_i \delta^n_i X, \delta^n_i X \right\rangle + \sum_{(2,n,p)} \left\langle r^n_i, \delta^n_i A \right\rangle   \notag \\
& =: I_1^{(n)} - I_2^{(n)} - I_3^{(n)} - \frac{1}{2} I_4^{(n)} + I_5^{(n)} + \frac{1}{2} I_6^{(n)} + I_7^{(n)} + I_8^{(n)} .  \label{2.2e}
\end{align}
Here the notation $D(t)$, $D_p(t)$, $\sum_{(1,n,p)}$, and $\sum_{(2,n,p)}$
are used with the same meaning as in the proof of Proposition~\ref{2.1m}.
We will consider the behavior of each term of the right-hand side of \eqref{2.2e} in the rest of this proof.

Since $D_p(t)$ is finite, it is easy to see that 
\begin{align*}
& \lim_{n \to \infty} I_1^{(n)} = \sum_{s \in D_p(t)} [\Delta f(A_s,X_s)] ,  \\
& \lim_{n \to \infty} I_2^{(n)} = \sum_{s \in D_p(t)} \sum_{k=1}^d \frac{\partial f}{\partial x_k}(A_{s-},X_{s-}) \Delta X_s^k ,  \\
& \lim_{n \to \infty} I_3^{(n)} = \sum_{s \in D_p(t)} \sum_{k=1}^m \frac{\partial f}{\partial a_k} (A_{s-},X_{s-}) \Delta A_s^k ,  \\
& \lim_{n \to \infty} I_4^{(n)} = \sum_{s \in D_p(t)} \sum_{k,l=1}^d \frac{\partial^2 f}{\partial x_k \partial x_l}(A_{s-},X_{s-}) \Delta X_s^k \Delta X_s^l.
\end{align*}
By the dominated convergence theorem, we have
\begin{equation*}
\lim_{n \to \infty} I_5^{(n)} = \sum_{k=1}^m \int_0 ^{t} \frac{\partial f}{\partial a_k}(A_{s-},X_{s-}) dA^k_s.
\end{equation*}
Lemma~\ref{2.1g} implies that
\begin{equation*}
\lim_{n \to \infty} I_6^{(n)} = \sum_{k,l=1}^d \int_0^{t} \frac{\partial f}{\partial x_k \partial x_l} (A_{s-},X_{s-}) d[X^k,X^l]_s.
\end{equation*}
By the definition of $\sum_{(2,n,p)}$ and the residual terms, we have
\begin{equation*}
\varlimsup_{n \to \infty} \left\lvert I_7^{(n)} \right\rvert \leq \omega \left( \nabla_x^2 f; \frac{2}{p} \right)  \sum_{k=1}^d \left[ X^k,X^k \right]_t, \quad \varlimsup_{n \to \infty} \left\lvert I_8^{(n)} \right\rvert \leq \omega \left( \nabla_a f; \frac{2}{p} \right)  \sum_{k=1}^m V(A^k)_t.
\end{equation*}

These observations imply that
\begin{equation*}
\varlimsup_{n \to \infty} \left\lvert (\text{RHS of \eqref{2.2e}}) - (\text{RHS of \eqref{2.2da}}) \right\rvert \leq \varphi(p).
\end{equation*}
Here $\varphi(p) \geq 0$ is a quantity that depends on $p$ and satisfies $\lim_{p \to \infty} \varphi(p) = 0$. 
Hence, the right-hand side of \eqref{2.2e} converges to that of \eqref{2.2da}.
This together with \eqref{2.2db} shows the convergence of all terms of \eqref{2.2e} except
\begin{equation*} %\label{2.2f}
\sum_{t_i \in \pi_n \cap [0,t]}  \left\langle \nabla_x f(A_{t^n_i},X_{t^n_i}), \delta^n_i X \right\rangle,
\end{equation*}
and the formula \eqref{2.2d}.
\end{proof}

\begin{Rem} \label{2.2fa}
The assumptions of Theorem~\ref{2.2c} can be relaxed.
Consider a $C^{1,2}$ function $f$ on a open set $U \subset \mathbb{R}^m \times \mathbb{R}^{d}$.
Suppose that, on each interval $[0,t]$, the function $(X,A) \in QV(\Pi;\mathbb{R}^d) \times FV_{\mathrm{loc}}^m$
takes its value in some compact subset of $U$.
Then the It{\^o}-F{\"o}llmer integral $\int_0^t \langle \nabla_x f(A_{s-},X_{s-}), dX_s \rangle$
exists and satisfies \eqref{2.2d}.
For example, we can apply the It{\^o}-F{\"o}llmer formula to $\log X_t$
if $X$ satisfies $\inf_{s \in [0,t]}X_{s} > 0$ for each $t$.
\end{Rem}

\begin{Rem} \label{2.2g}
If a path $(A,X)$ is continuous, the convergence of the It{\^o}-F{\"o}llmer integral in Theorem~\ref{2.2c}
holds uniformly on compacts.
See
%Davis, Ob{\l}{\'o}j, and Siorpaes~
\cite[Theorem 2.5]{Davis_Obloj_Siorpaes_2017}.
\end{Rem}

As an application of Theorem~\ref{2.2c}, we have the following integration by parts formula
for the It{\^o}-F{\"o}llmer integral.

\begin{Cor} \label{2.2h}
Let $(X,Y) \in QV(\Pi;\mathbb{R}^2)$.
If one of the It{\^o}-F{\"o}llmer integrals $\int_{0}^t Y_{s-} dX_s$ and $\int_{0}^t X_{s-} dY_s$ exist,
then the other also exists and they satisfy
\begin{equation} \label{2.2i}
X_t Y_t - X_0 Y_0 = \int_0^t Y_{s-} dX_s + \int_0^t X_{s-} dY_s + [X,Y]_t, \quad t \geq 0.
\end{equation}
\end{Cor}

\begin{proof}
It suffices to prove the case where $\int_{0}^t Y_{s-} dX_s$ exists.
Applying Theorem~\ref{2.2c} to $f(x,y) = xy$, we get
\begin{equation*}
X_t Y_t - X_0 Y_0 = \int_0^t \langle \nabla f(X_{s-},Y_{s-}) ,d (X, Y)_s \rangle + [X,Y]_t, \quad t \geq 0.
\end{equation*}
Moreover, we have
\begin{equation*}
\int_{0}^{t} X_{s-} dY_s = \int_0^t \langle \nabla f(X_{s-},Y_{s-}) ,d (X, Y)_s \rangle - \int_{0}^{t} Y_{s-} dX_s,
\end{equation*}
because both of the It{\^o}-F{\"o}llmer integrals $\int_{0}^t Y_{s-} dX_s$ and $\int_0^t \langle \nabla f(X_{s-},Y_{s-}) ,d (X, Y)_s \rangle$ exist.
Thus we obtain \eqref{2.2i}
\end{proof}

\subsection{It{\^o}-F{\"o}llmer integrals of admissible integrands}

In this subsection, we study some more basic properties of
the It{\^o}-F{\"o}llmer integrals with respect to $X$ for integrands
of the form $\nabla_x f(A_{t},X_{t})$.

\begin{Def} \label{2.3a}
Let $X \in QV(\Pi; \mathbb{R}^d)$.
A $d$-dimensional {\cadlag} path $\xi$ is called an \emph{admissible integrand of $X$}
if, for any $T>0$, there exist $m \in \mathbb{N}$, a function $f$ of $C^{1,2}$-class defined on a subset $U$ of $\mathbb{R}^m \times \mathbb{R}^d$,
and $A \in FV_{\mathrm{loc}}^m$ such that
$\xi_t = \nabla_x f(A_t,X_t)$ holds for all $t \in [0,T]$
and the image of the map $[0,T] \ni t \mapsto (A_t,X_t) \in U$ is relatively compact in $U$.
\end{Def}

For the terminology ``admissible integrand'', we follow
%Schied~
\cite{Schied_2014}.

\begin{Rem} \label{2.3b}
\begin{enumerate}
\item For a $d$-dimensional admissible integrand $\xi = (\xi^1,\dots,\xi^d)$
	of $X \in QV(\Pi;\mathbb{R}^d)$, the equation
	\begin{equation*}
	\int_0^t \left\langle \xi_{s-}, dX_s \right\rangle = \sum_{i=1}^d \int_0^t \xi^i_{s-} dX^i_s
	\end{equation*}
	is not necessarily valid.
	Note that the existence of the It{\^o}-F{\"o}llmer integral of the left-hand side 
	does not imply the existence of those of the right-hand side.
\item If $\xi$ and $\eta$ are admissible integrands of $X \in QV(\Pi;\mathbb{R}^d)$,
	their linear combination $\alpha \xi + \beta \eta$ ($\alpha,\beta \in \mathbb{R}$) is again
	an admissible integrand that satisfies
	\begin{equation} \label{2.3c}
	\int_0^t \left\langle \alpha \xi_{s-}+ \beta \eta_{s-},dX_s \right\rangle
	= \alpha \int_0^t \left\langle \xi_{s-},dX_s \right\rangle + \beta \int_0^t \left\langle \eta_{s-},dX_s \right\rangle
	\end{equation}
	for all $t \in [0,T]$.
	In other words, the space of admissible integrands of $X$ is a vector space and
	the map $\xi \mapsto \int_{0}^{t} \xi_{s-} dX_s$ on it is linear.
\item Let $X \in QV(\Pi)$, and $\xi_t = g(A_t,X_t)$ ($t \geq 0$), where $g \in C^1(\mathbb{R}^m \times \mathbb{R})$
	and $A \in FV_{\mathrm{loc}}^m$. Then $\xi$ is an admissible integrand of $X$.
	Indeed, $\xi_t = \frac{\partial}{\partial x} f(A_t,X_t)$ holds for $f(a,x) = \int_0^x g(a,y) dy$.
\end{enumerate}
\end{Rem}

\begin{Rem} \label{2.3d}
Let $\xi$ be an admissible integrand of $X \in QV(\Pi;\mathbb{R})$ with a 
representation $\xi_t = \frac{d}{dx}f(B_{t},X_{t})$ for $t \in [0,T]$.
Then, for any $A \in FV_{\mathrm{loc}}$, the path $\xi$ is an admissible integrand of $Y := X+A$ and satisfies, for all $t \in [0,T]$,
\begin{equation} \label{2.3e}
\int_0^t \xi_{s-} dY_s = \int_0^t \xi_{s-} dX_s + \int_0^t \xi_{s-} dA_s.
\end{equation}
In fact, set $F(a,b,y) = f(b,y-a)$, and we obtain 
a representation $\xi_t = \frac{\partial F}{\partial y}(A_t,B_t,Y_t)$.
\end{Rem}

We will now study the quadratic variation of a path given by It{\^o}-F{\"o}llmer integration.

\begin{Prop} \label{2.3f}
Let $X \in QV(\Pi;\mathbb{R}^d)$ and let 
$\xi: \mathbb{R}_{\geq 0} \to \mathbb{R}^d$ be an admissible integrand of $X$.
Define $Y_t = \int_0^t \langle \xi_{s-},dX_s \rangle$ for $t \in \mathbb{R}_{\geq 0}$.
Then, it holds that $Y \in QV(\Pi;\mathbb{R})$ and 
\begin{equation*}
[Y,Y]_t = \sum_{k,l=1}^d \int_{0}^{t} \xi^{k}_{s-} \xi^{l}_{s-} d[X^k,X^l]_s.
\end{equation*}
\end{Prop}

\begin{proof}
Fix $T>0$ and an expression $\xi = (\nabla_x f) \circ (A,X)$ on $[0,T]$, as in Definition~\ref{2.3a}.
Define
\begin{equation}
B_t = Y_t - f(A_t,X_t) - f(A_0,X_0), \quad t \in [0,T].
\end{equation}
Then, by Theorem~\ref{2.2c}, $B$ is a {\cadlag} path of locally finite variation whose jumps are given by
$$ \Delta B_t = \Delta f(A_t,X_t) - \sum_{i=1}^d \frac{\partial f}{\partial x_i}(A_{t-},X_{t-}) \Delta X^i_t. $$
Using Proposition~\ref{2.1m} and Corollary~\ref{2.1z}, we can directly calculate
the quadratic variation of $Y$ as follows.
\begin{align*}
\lbrack Y,Y \rbrack_t
& = \lbrack f(A,X),f(A,X) \rbrack_t + \sum_{0 \leq s \leq t}\left[ (\Delta B_s)^2 - 2 \Delta B_s \Delta f(A_s,X_s)\right] \\
& = \sum_{i,j=1}^d \int_0^t \left( \frac{\partial f}{\partial x_i} \frac{\partial f}{\partial x_j} \right) (A_{s-},X_{s-}) d[X^i,X^j]_s.
\end{align*}
Since $T > 0$ is arbitrarily fixed, this formula is extended to $\mathbb{R}_{\geq 0}$.
\end{proof}

The next proposition is a multi-dimensional version of Proposition~\ref{2.3f}.

\begin{Prop} \label{2.3g}
Let $X \in QV(\Pi; \mathbb{R}^d)$, let $\xi^{(1)},\dots,\xi^{(m)}$ be admissible integrands of $X$,
and let $Y^l_t = \int_0^t \left\langle \xi_{s-}^{(l)},dX_s \right\rangle$ for all $t \geq 0$ and $l \in \{1,\dots, m \}$. 
Then $Y=(Y^1,\dots,Y^m)$ belongs to $QV(\Pi; \mathbb{R}^m)$ and its quadratic covariations have the following form.
\begin{equation} \label{2.3gb}
\left\lbrack Y^k,Y^l \right\rbrack_t = \sum_{i,j =1}^d \int_{0}^{t} \xi^{(k),i}_{s-} \xi^{(l),j}_{s-} d\!\left\lbrack X^i,X^j \right\rbrack_s.
\end{equation}
\end{Prop}

\begin{proof}
Applying Proposition~\ref{2.3f} to $Y^{k}$ and $Y^{k} + Y^{l}$
for $k,l \in \{1, \dots m \}$, we obtain \eqref{2.3gb}.
\end{proof}

By Propositions \ref{2.3f} and \ref{2.3g}, we see that a path of the form
$Y_t = \int_0^t \langle \xi_{s-}, dX_s \rangle$ again has quadratic variation.
So we can consider It{\^o}-F{\"o}llmer integration by $Y$.
%This integration is expected to satisfy
%the associativity law, as usual It{\^o} integrals do.
The associativity of It{\^o}-F{\"o}llmer integration with respect to continuous integrators
are proved in \cite[Theorem 13]{Schied_2014}.
In this paper, we will extend this result to {\cadlag} integrators.
For this purpose, we prepare a lemma.
This lemma is due to
%Schied~
\cite[Proof of Theorem 13]{Schied_2014}
for the continuous path case.

\begin{Lem} \label{2.3h}
Let $X \in QV(\Pi;\mathbb{R}^d)$, let $\xi$ be an admissible integrand of $X$,
and let $Y_t =  \int_0^t \left\langle \xi_{s-},dX_s \right\rangle$.
Then, for each $T > 0$, we can choose a $C^{1,2}$ function $F$ such that
$\xi_t  = \nabla_x F(A_t,X_t)$, $Y_t = F(A_t,X_t)$, and
\begin{equation}  \label{2.3i}
\sum_{i=1}^m \int_0^t \frac{\partial F}{\partial a_i}(A_{s-},X_{s-})d(A^i)^{\mathrm{c}}_{s} + \frac{1}{2} \sum_{i,j = 1}^d \int_0^t \frac{\partial^2 F}{\partial x_i \partial x_j} (A_{s-},X_{s-}) d[X^i,X^j]^{\mathrm{c}}_s = 0
\end{equation}
holds for all $t \in [0,T]$.
\end{Lem}

\begin{proof}
Fix $T>0$ arbitrarily.
Take $\widetilde{m} \in \mathbb{N}$, $\widetilde{A} \in FV_{\mathrm{loc}}^{m}$, 
an open set $U \subset \mathbb{R}^{m} \times \mathbb{R}^d$, and 
$f \in C^{1,2}(U,\mathbb{R})$ such that 
$\xi_t = \nabla_x f(\widetilde{A}_t,X_t)$ for $t \in [0,T]$.
Let us define $\widetilde{A}_t^{0} = f(\widetilde{A}_t,X_t) -  f(\widetilde{A}_0,X_0) - Y_t$.
Then, the path $\widetilde{A}^{0}$ is of finite variation on $[0,T]$ by It{\^o}'s formula.
Moreover, let $m = \widetilde{m}+1$,
$A = (\widetilde{A}^0, \widetilde{A})$, and
\begin{equation*}
F(\widetilde{a}_0,\widetilde{a},x) = f(\widetilde{a},x) - f(A_0,X_0) - \widetilde{a}_0, \quad (\widetilde{a}_0,\widetilde{a},x) \in \mathbb{R} \times U.
\end{equation*}
Then, we have $F \in C^{1,2}(\mathbb{R} \times U)$, $Y_t = F(A_t,X_t)$, and
\begin{equation*}
\nabla_x F(A_t,X_t) = \nabla_x f(\widetilde{A}_t,X_t) = \xi_t, \quad t \in [0,T].
\end{equation*}
By construction and It{\^o}'s formula, $F$ clearly satisfies \eqref{2.3i}.
\end{proof}

The following is the main theorem of this subsection.

\begin{Th} \label{2.3m}
Let $X \in QV(\Pi;\mathbb{R}^d)$ and
$Y_t^k =  \int_0^t \left\langle \xi_{s-}^{(k)},dX_s \right\rangle$,
where $\xi^{(1)},\dots,\xi^{(\nu)}$ are admissible integrands of $X$.
Suppose that $\eta = (\eta^1,\dots,\eta^\nu)$
is an $\nu$-dimensional {\cadlag} path.
Then, the It{\^o}-F{\"o}llmer integral $\int_0^t \langle \eta_{s-}, dY_s \rangle$
exists if and only if the It{\^o}-F{\"o}llmer integral
$\int_0^t \langle \sum_{k=1}^\nu \eta^k_{s-} \xi_{s-}^{(k)}, dX_s \rangle$ exists.
If one of them exists, they satisfy
\begin{equation} \label{2.3n}
\int_0^t \left\langle \eta_{s-},d Y_s \right\rangle = \int_0^t \left\langle \sum_{k=1}^\nu \eta^k_{s-} \xi_{s-}^{(k)}, dX_s \right\rangle, \quad t \in \mathbb{R}_{\geq 0}.
\end{equation}
\end{Th}

\begin{proof}
Fix $T > 0$.
For each $l \in \{ 1,\dots, \nu\}$,
let $Y^{l} = F^{l}(A^{l},X)$ be the expression of Lemma~\ref{2.3h}
where $A^{l} \in FV_{\mathrm{loc}}^{m_l}$.
By the second-order Taylor expansion, we have
\begin{align*}
\delta^n_i F^{l}(A^{(l)},X)
& = \left\langle \nabla_{a} F^{l}(A^{(l)}_{t^n_i},X_{t^n_i}), \delta^n_i A^{(l)} \right\rangle + \left\langle r^{n,l}_i,\delta^n_i A^{(l)} \right\rangle + \left\langle \nabla_x F^{l}(A^{(l)}_{t^n_i},X_{t^n_i}), \delta^n_i X \right\rangle  \notag \\
& \quad + \frac{1}{2} \left\langle \nabla_x^2 F^{l}(A^{(l)}_{t^n_i},X_{t^n_i})  (\delta^n_i X), \delta^n_i X \right\rangle + \left\langle R^{n,l}_i (\delta^n_i X), \delta^n_i X \right\rangle,
\end{align*}
where $r^{n,l}_i$ and $R^{n,l}_i$ are corresponding residual terms.
(See the proof of Theorem~\ref{2.2c} for the specific forms.)
Let $D^{l}(t)$, $D^{l}_p(t)$, $\sum_{(1,n,p)_l}$, and $\sum_{(2,n,p)_l}$ be
the quantities in the proof of Proposition~\ref{2.1m} corresponding to the path $(A^{(l)},X)$.
Then,
\begin{align*}
& \sum_{\pi_n \cap [0,t]} \langle \eta_{t^n_i}, \delta^n_i Y \rangle - \sum_{\pi_n \cap [0,t]}\left\langle  \sum_{l=1}^{\nu} \eta^{l}_{t_i} \xi^{(l)}_{t_i}, \delta^n_i X \right\rangle  \notag \\
& = \sum_{l=1}^{\nu} \sum_{(1,n,p)_l}  \eta^{l}_{t^n_i} \delta^n_i F^{l}( A^{(l)},X ) -  \sum_{l=1}^{\nu} \sum_{(1,n,p)_l} \eta^l_{t_i} \left\langle \nabla_x F^{l} (A^{(l)}_{t^n_i},X_{t^n_i} ), \delta^n_i X \right\rangle  \notag \\ 
& \quad + \sum_{l=1}^{\nu} \sum_{\pi_n \cap [0,t]} \eta^{l}_{t_i}  \left\langle \nabla_{a} F^{l}(A^{(l)}_{t^n_i},X_{t^n_i}) , \delta^n_i A^{(l)} \right\rangle  - \sum_{l=1}^{\nu} \sum_{(1,n,p)_l} \eta^{l}_{t^n_i}  \left\langle \nabla_{a} F^{l}(A^{(l)}_{t^n_i},X_{t^n_i}) , \delta^n_i A^{(l)} \right\rangle    \notag \\
& \quad + \frac{1}{2} \sum_{l=1}^{\nu} \sum_{\pi_n \cap [0,t]} \left\langle \eta^l_{t^n_i} \nabla_x^2 F^{l}(A^{(l)}_{t^n_i},X_{t^n_i}) ( \delta^n_i X), \delta^n_i X \right\rangle   \notag \\
& \quad - \frac{1}{2} \sum_{l=1}^{\nu} \sum_{(1,n,p)_l} \left\langle \eta^{l}_{t^n_i} \nabla_x^2 F^{l}(A^{(l)}_{t^n_i},X_{t^n_i}) (\delta^n_i X), \delta^n_i X \right\rangle  \notag \\
& \quad + \sum_{l=1}^{\nu}  \sum_{(2,n,p)_l} \eta^l_{t^n_i} \left\langle r^{n,l}_i, \delta^n_i A^{(l)} \right\rangle + \sum_{l=1}^{\nu} \sum_{(2,n,p)_l} \eta_{t^n_i}^l \left\langle R^n_i ( \delta^n_i X), \delta^n_i X \right\rangle   \notag \\
& =: I_1^{(n)} - \widetilde{I}_1^{(n)} + I_2^{(n)} - \widetilde{I}_{2}^{(n)} + \frac{1}{2} I_3^{(n)} - \frac{1}{2} \widetilde{I}_{3}^{(n)} + I_4^{(n)} + I_{5}^{(n)}.
\end{align*}

Clearly we have
\begin{equation*} 
\lim_{n \to \infty} \left( I_1^{(n)} - \widetilde{I}_1^{(n)}  \right) = \sum_{l=1}^{\nu} \sum_{s \in D_1^{l}(p)}  \eta^{l}_{s-} \left( \Delta Y^{l}_{s} - \sum_{k=1}^d \xi_{s-}^{(l),k} \Delta X^k_{s} \right) = 0.
\end{equation*}
By the dominated convergence theorem, we see that
\begin{align*}
\lim_{n \to \infty} I_2^{(n)}
& = \sum_{l=1}^{\nu} \sum_{j=1}^{m_l} \int_{0}^{t} \eta^{l}_{s-} \frac{\partial F^{l}}{\partial a_j} (A^{(l)}_{s-},X_{s-}) dA^{(l),j}_s   \notag \\
\lim_{n \to \infty} \widetilde{I}_2^{(n)}
& = \sum_{l=1}^{\nu} \sum_{j=1}^{m_l} \sum_{s \in D_p^{l}(p)} \eta^{l}_{s-} \frac{\partial F^l}{\partial a_j}(A^{(l)}_{s-},X_{s-}) \Delta A^{(l)}_{s}.
\end{align*}
By Lemma~\ref{2.1g}, we have
\begin{align*} 
\lim_{n \to \infty} I_3^{(n)} 
& = \sum_{l=1}^{\nu} \sum_{j,k = 1}^{d} \int_{0}^{t} \eta^{l}_{s-} \frac{\partial^2 F^{l}}{\partial x_j \partial x_k} (A^{(l)}_{s-},X_{s-}) d[X^j,X^k]_s   \notag \\
\lim_{n \to \infty}  \widetilde{I}_3^{(n)} & =  \sum_{l=1}^{\nu} \sum_{j,k = 1}^{d} \sum_{s \in D^l_p(t)} \eta^{l}_{s-} \frac{\partial^2 F^{l}}{\partial x_j \partial x_k} (A^{(l)}_{s-},X_{s-}) \Delta X^j_s \Delta X^k_s.
\end{align*}
By a discussion similar to that of the proof of Theorem~\ref{2.2c},
we can take a constant $C$, which does not depend on $\varepsilon$ and $n$, such that
\begin{equation*}
\varlimsup_{n \to \infty} \left\lvert I_4^{(n)} \right\rvert \leq  C \sum_{l=1}^{\nu}\omega\left(\nabla_a F^l; \frac{2}{p} \right), \quad 
\varlimsup_{n \to \infty} \left\lvert I_5^{(n)} \right\rvert \leq  C \sum_{l=1}^{\nu} \omega\left( \nabla^2_x F^l ; \frac{2}{p} \right) .
\end{equation*}
As a consequence of these observations and \eqref{2.3i} we obtain
\begin{align*}
& \varlimsup_{n \to \infty} \left\lvert \sum_{\pi_n \cap [0,t]} \langle \eta_{t^n_i},\delta^n_i Y \rangle - \sum_{\pi_n \cap [0,t]}\left\langle  \sum_{l=1}^{\nu} \eta^{l}_{t^n_i} \xi^{(l)}_{t^n_i}, \delta^n_i X \right\rangle \right\rvert \notag \\
& \leq C \sum_{l=1}^{\nu} \left( \omega \left( \nabla_a F^{l}; \frac{2}{p}\right) + \omega \left( \nabla^2_x F^{l} ; \frac{2}{p}\right) \right) + \left\lvert \sum_{l=1}^{\nu} \sum_{s \in D^{l}(t) \setminus D_p^{l}(t)} \sum_{j=1}^{m_l}  \eta^{l}_{s-} \frac{\partial F^l}{\partial a^j}(A^{(l)}_{s-},X_{s-}) \Delta A^{(l),j}_{s} \right\rvert \notag \\
& \quad + \frac{1}{2} \left\lvert \sum_{l=1}^{\nu} \sum_{s \in D^{l}(t) \setminus D_p^{l}(t)} \sum_{j,k = 1}^{d} \eta^{l}_{s-} \frac{\partial^2 F^{l}}{\partial x_j \partial x_k} (A^{(l)}_{s-},X_{s-}) \Delta X^j_s \Delta X^k_s  \right\rvert.
\end{align*}
Letting $p \to \infty$, we see that the right-hand side of this inequality converges to 0.
Hence, if one of the two It{\^o}-F{\"o}llmer integrals exists, then the other also exists and their values are equal.
\end{proof}

\begin{Rem} \label{2.3nf}
If $\eta = (\eta^1,\dots,\eta^{\nu})$ is an admissible integrand of $Y$,
associativity in the sense of \eqref{2.3n} can be shown by direct calculation.

Fix $T>0$ and let $\xi_t^{(l)} = \nabla_x F^{l}(A^{(l)}_{t},X_t)$ be
those in the proof of Theorem~\ref{2.3m}.
Define $m = m_1 + \dotsb + m_\nu$, 
$a = (a^{(1)}, \dots, a^{(\nu)}) \in \mathbb{R}^{m}$,
$\mathbf{F}(a,x) = (F^1(a^{(1)},x),\dots,F^{\nu}(a^{(\nu)},x))$,
and $A_t = (A^{(1)}_t,\dots,A^{(\nu)}_t)$.
Then we have $Y_t = \mathbf{F}(A_t,X_t)$.

Choose $B \in FV_{\mathrm{loc}}^{n}$ and $g$ of $C^{1,2}$-class such that
$\eta_t = \nabla_y g(B_t,Y_t)$ holds for all $t \in [0,T]$.
Let $C = (B,A)$ and $h(c,x) = h(b,a,x) = g(b,\mathbf{F}(a,x))$.
Then, we have
\begin{equation}
\nabla_x h (C_t,X_t) = \nabla_y g(B_t,Y_t) \nabla_x \mathbf{F}(A_t,X_t) = \sum_{l=1}^\nu \eta^l_t \xi^{(l)}_t .
\end{equation}
Hence, $\sum_{l=1}^\nu \eta^l_t \xi^{(l)}_t$ is an admissible integrand of $X$.
This expression is given by
%Schied~
\cite{Schied_2014} in the proof of Theorem 13.

Applying the It{\^o} formula to $h(C_t,X_t)$, we get
\begin{align} \label{2.3q}
\int_0^t \left\langle \sum_{l=1}^\nu \eta^l_{s-} \xi^{(l)}_{s-},dX_s \right\rangle 
& = h (C_t,X_t) - h (C_0,X_0) - \sum_{i=1}^{n+m} \int_{]0,t]} \frac{\partial h}{\partial c_i}(C_{s-},X_{s-}) d(C^i)^{\mathrm{c}}_s  \notag \\
& \quad - \frac{1}{2} \sum_{i,j=1}^d \int_{0}^{t} \frac{\partial^2 h}{\partial x_i \partial x_j}(C_{s-},X_{s-}) d\left[ X^i,X^j \right]^{\mathrm{c}}_s  \notag \\
& \quad - \sum_{0 < s \leq t} \left\{ \Delta h (C_s,X_s) - \sum_{i=1}^d \frac{\partial h}{\partial x_i}  (B_{s-},X_{s-}) \Delta X^i_s \right\}.
\end{align}
On the other hand, applying the It{\^o} formula to $g(B,Y)$, we see that
\begin{align} \label{2.3r}
\int_0^t \left\langle \eta_{s-},dY_s \right\rangle
%& = \int_0^t \left\langle \nabla_y g(B_{s-},Y_{s-}),dY_s \right\rangle  \notag \\
& = g(B_t,Y_t) - g(B_0,Y_0) - \sum_{k=1}^{m} \int_{0}^{t} \frac{\partial g}{\partial b_k}(B_{s-},Y_{s-}) d(B^k)^{\mathrm{c}}_s  \notag \\
& \quad - \frac{1}{2} \sum_{k,l=1}^\nu \int_{0}^{t} \frac{\partial^2 g}{\partial y_k \partial y_l}(B_{s-},Y_{s-}) d\left[ Y^k,Y^l \right]^{\mathrm{c}}_s  \notag \\
& \quad - \sum_{0 < s \leq t} \left\{ \Delta g(B_s,Y_s) - \sum_{l=1}^\nu \frac{\partial g}{\partial y_l}  (B_{s-},Y_{s-}) \Delta Y^l_s \right\} .
\end{align}
Thanks to Proposition~\ref{2.3g} and \eqref{2.3i}, we can check that the right sides of \eqref{2.3q} and \eqref{2.3r} are equal.
\end{Rem}

The following integration by parts formula is proved as a corollary of associativity.

\begin{Cor} \label{2.3s}
Let $X \in QV(\Pi)$.
For admissible integrands $\xi,\eta$ of $X$ and for $A,B \in FV_{\mathrm{loc}}$, we define
\begin{equation*}
Y_t = \int_0^t \xi_{s-} dX_s +A_t , \quad Z_t = \int_0^t \eta_{s-} dX_s + B_t.
\end{equation*}
Then we have
\begin{align*}
Y_t Z_t = \int_0^t Y_{s-} dZ_s + \int_0^t Z_{s-} dY_s + [Y,Z]_t .
\end{align*}
\end{Cor}

\begin{proof}
According to Corollary~\ref{2.2h},
what we need to show is the existence of the It{\^o}-F{\"o}llmer integral
$\int_0^t Y_{s-} dZ_s$.
Theorem~\ref{2.3m} yields that this condition is equivalent to the existence of
the following two It{\^o}-F{\"o}llmer integrals.
$$ \int_0^t Y_{s-} \eta_{s-} dX_s, \quad \int_0^t Y_{s-} dB_s. $$
We can see that the first integral exists
because $\xi_t Y_t$ is an admissible integrand of $X$,
which is shown by using the expression for $Y$ in Lemma~\ref{2.3h}.
The existence of the second integral is obvious since $B \in FV_{\mathrm{loc}}$.
\end{proof}

\subsection{Quadratic variations of paths and semimartingales}

In this subsection,
we describe relations between quadratic variations in F{\"o}llmer's sense and those of semimartingales.
Let $(\Omega,\mathcal{F},(\mathcal{F}_t)_{t \in \mathbb{R}_{\geq 0}})$ be a filtered measurable space
such that $(\mathcal{F}_t)$ is right continuous and universally complete,
and let $\mathcal{P}$ be a family of probability measures on $(\Omega,\mathcal{F})$.

\begin{Def} \label{5.1a}
Let $\pi = (\tau_k)_{k \in \mathbb{N}}$ be a sequence of random variables
such that $0 = \tau_0 \leq \tau_1 \leq \dots$.
If for every $\omega$ the partition $(\tau_k(\omega))_{k \in \mathbb{N}}$
satisfies the conditions in Section 2.1, 
$\pi$ is called a \emph{random partition}.
In addition, if each $\tau_k$ is an $(\mathcal{F}_t)$-stopping time,
we call $\pi$ an \emph{optional partition}.
We denote the sequence $(\tau_k(\omega))_{k \in \mathbb{N}}$ by $\pi(\omega)$ for each $\omega \in \Omega$.
\end{Def}

For a sequence of optional partitions $\Pi = (\pi_n)_{n \in \mathbb{N}}$,
each partition is denoted by $\pi_n = (\tau^n_k)_{k \in \mathbb{N}}$, for example.
Given an optional partition $\pi$ and a {\cadlag} process $X$,
the oscillation of paths of $X$ along $\pi$ is defined by 
\begin{equation*}
O_t(X(\omega),\pi(\omega)) = \sup_{t_i \in \pi} \sup\{ \lVert X_s(\omega) - X_u(\omega) \rVert_{\mathbb{R}^d} \mid s,u \in [t_i,t_{i+1}\mathclose{[} \cap [0,t] \}.
\end{equation*}
Then, we can easily check that $(\omega,t) \mapsto O_t(X(\omega),\pi(\omega))$ is a {\cadlag} adapted process
because $X$ is {\cadlag} and adapted, and $\tau^n_k$'s are stopping times.
Also, recall that if a sequence $(g_n)$ in $L^p(P)$ satisfies
$\sum_{n \in \mathbb{N}} \lVert g_n - g \rVert_{L^p(P)}^p < +\infty$, 
then $(g_n)$ converges almost surely to $g$.

Suppose that $X$ is a semimartingale under $P \in \mathcal{P}$.
Let us consider the usual quadratic variation $[X,X]^P$ of a semimartingale $X$.
It is a well known fact that if a sequence of optional partitions
$\Pi = (\pi_n)$ satisfies $\lvert \pi_n(\omega) \rvert \to 0$ almost surely,
the sequence $[X(\omega),X(\omega)]^{\pi_{n}(\omega)}$ converges to $[X,X]^P$ in the ucp topology.
Hence, by taking a proper subsequence, $P$-almost all paths of $X$ have quadratic variation along
$\pi(\omega)$ and they are almost surely equal to $[X,X]^P$.

Our aim is to find a sufficient condition for $\Pi$ under which 
$P$-almost every path of $X$ has quadratic variation along $\Pi$ without taking a subsequence.
For this we introduce some spaces of processes and the norms on them.
Let $p \in [1,\infty[$.
A {\cadlag} adapted process $X$ belongs to $\mathcal{S}^p(P)$ if it satisfies
$\lVert X \rVert_{\mathcal{S}^p(P)} := \lVert \sup_{t \geq 0} \lvert X_t \rvert \rVert_{L^p(P)} < \infty$.
Then the space $(\mathcal{S}^p(P),\lVert ~~ \rVert_{\mathcal{S}^p(P)})$ is a Banach space.
Let $\mathcal{V}$ be the space of {\cadlag} adapted processes of locally finite variation
with initial value 0,
and let $\mathcal{M}_{\mathrm{loc},0}(P)$ be the space of {\cadlag} local martingales starting from 0.
The $\mathcal{H}^p$-norm of a semimartingale $X$ with respect to $P$ is
\begin{gather*}
\lVert X \rVert_{\mathcal{H}^p(P)} = \lVert X_0 \rVert_{L^p(P)} + \inf_{\substack{X=X_0 + M + A \\ M \in \mathcal{M}_{\mathrm{loc},0}, A \in \mathcal{V}}} j_p(M,A),    \\
j_p(M,A) =  \left\lVert \lbrack M,M \rbrack_{\infty}^{1/2} \right\rVert_{L^p(P)} + \left\lVert V(A)_{\infty} \right\rVert_{L^p(P)},
\end{gather*}
where $V(A)$ is the total variation process of $A$.
The space consisting of semimartingales with finite $\mathcal{H}^p$ norm is denoted by $\mathcal{H}^p(P)$.
Then the space $(\mathcal{H}^p(P),\lVert~~ \rVert_{\mathcal{H}^p(P)})$ is Banach.
Basic properties of these spaces are explained in, for example,
%Emery~
\cite{Emery_1978,Emery_1979} and
%Protter~
\cite[Chapter V]{Protter_2005}.
$\mathcal{S}^p$ and $\mathcal{H}^p$ norms have the following relation: there exists a positive constant $C_p$,
that only depends on $p$, such that 
\begin{equation} \label{5.1c}
\left\lVert X \right\rVert_{\mathcal{S}^p(P)} \leq C_p \lVert X \rVert_{\mathcal{H}^p(P)}
\end{equation}
holds for all semimartingale $X$ of $\mathcal{H}^p(P)$.

Given a class of stochastic process, denoted by $\mathcal{C}$, we say that a {\cadlag}
process $X$ prelocally belongs to $\mathcal{C}$
if there is a localizing sequence of stopping times $T_n$ such that
\begin{equation*}
X^{T_n-} := X 1_{\lBrack 0,T_n \lBrack} + (X_{T_n -}) 1_{\lBrack T_n,\infty \lBrack} \in \mathcal{C}
\end{equation*}
for all $n$. 
Here, a stochastic interval of the type $\lBrack S,T \lBrack$ is defined by $\lBrack S,T \mathclose{\lBrack} = \{ (\omega,t) \in \Omega \times \mathbb{R}_{\geq 0} \mid S(\omega) \leq t < T(\omega) \}$.
Recall that for any semimartingale $X$, the process $X-X_0$ prelocally belongs to
$\mathcal{H}^p(P)$ for all $p \in [1,\infty[$.
The following Proposition is due to
%Davis, Ob{\l}{\'o}j, and Siorpaes~
\cite[Proposition 2.4]{Davis_Obloj_Siorpaes_2017} for continuous semimartingales.
Its proof can directly applied to {\cadlag} semimartingales.

\begin{Prop} \label{5.1f}
Let $P \in \mathcal{P}$,
let $X$ be a $P$-semimartingale, 
and let $\Pi = (\pi_n)$ be a sequence of optional partitions
such that $\lvert \pi_n(\omega) \rvert \to 0$ $P$-a.s..
Suppose that for all $T$ the condition $\sum_{n} O_T(X(\omega),\pi_n(\omega)) < +\infty$ holds
$P$-almost surely.
Then, for $P$-almost every $\omega$, the path $X(\omega)$ belongs to $QV(\Pi(\omega))$ and satisfies
$\lbrack X(\omega),X(\omega) \rbrack^{\Pi(\omega)} = \lbrack X,X \rbrack^P(\omega)$.
\end{Prop}

\begin{proof}
We can assume $X_0 = 0$ without loss of generality.
Recall that $X$ satisfies
\begin{equation*}
[X,X]^P = X^2 - X_{-} \bullet X 
\end{equation*}
where the last term of the right hand wide denotes the stochastic integral with respect to $P$.
First we suppose $X \in \mathcal{H}^4(P)$.
Fix an arbitrary $T$ in $\mathbb{R}_{\geq 0}$ and define $K_T := \sum_{n} O_T(X,\pi_n)$.
Then, $K_T$ is a random variable since it is defined as the (convergent)
series of a sequence of positive random variables.
We define a probability measure $Q$ equivalent to $P$ by 
\begin{equation*}
Q(A) = \frac{1}{E[\exp (-K_T)]} \int_A e^{-K_T(\omega)} P(d\omega).
\end{equation*}
Then, $X$ is again a semimartingale with respect to $Q$.
Here recall that the stochastic integral is invariant under an equivalent change of probability measure.
(See, for example, 
%He, Wang, and Yan~
\cite[12.22 Theorem]{He_Wang_Yan_1992}.)
We have $K_T \in L^4(Q)$, and $X \in \mathcal{H}^4(Q)$ by definition of $Q$.
%\footnote{Now the density $\frac{dQ}{dP}$ is bounded.}.
Let 
\begin{equation*}
H^n := \sum_{k} X_{\tau_k^n} 1_{\rBrack \tau_k^n, \tau_{k+1}^n \rBrack}, \quad K^n := H^n - X_-.
\end{equation*}
Then $H^n$ and $K^n$ are locally bounded predictable processes.
By definition, we see that
\begin{equation} \label{5.1g}
\sum_{n \in \mathbb{N}} \sup_{t \in [0,T]} \lvert K^n_t \rvert^2 \leq \left( \sum_{n \in \mathbb{N}} \sup_{t \in [0,T]} \lvert K^n_t \rvert \right)^2 \leq K_T^2 \in L^2(Q).
\end{equation}
Consider a decomposition $X = M + A$ such that
$M \in \mathcal{M}_{\mathrm{loc},0}(Q)$ and $A \in \mathcal{V}$.
Since $K^n$ is locally bounded, the decomposition $K^n \bullet X = K^n \bullet M + K^n \bullet A$
satisfies $K^n \bullet M \in \mathcal{M}_{\mathrm{loc},0}(Q)$ and $K^n \bullet A \in \mathcal{V}$.
Combining \eqref{5.1c}, \eqref{5.1g}, the monotone convergence theorem and Schwarz's inequality, 
we get
\begin{align*}
& \sum_{n \in \mathbb{N}} \left\lVert (K^n \bullet X)^T \rvert \right\rVert_{\mathcal{S}^2(Q)}^2 \leq C \left\lVert K_T \right\rVert_{L^4(Q)}^2 \left( \left\lVert \lbrack M,M \rbrack_T^{1/2} \right\rVert_{L^4(Q)}^2 + \left\lVert V(A)_T \right\rVert_{L^4(Q)}^2 \right)
\end{align*}
where $C$ is a positive constant.
Since $X \in \mathcal{H}^4(Q)$, we can take a decomposition such that
$\lVert \lbrack M,M \rbrack_T^{1/2} \rVert_{L^4(Q)}$
and $\lVert V(A)_T \rVert_{L^4(Q)}^2$
are finite.
Hence $ \sum_{n} \left\lVert (K^n \bullet X)^T \rvert \right\rVert_{\mathcal{S}^2(Q)}^2$
converges.
This proves that $(K^n \bullet X)^*_T \to 0$ holds $Q$-almost surely, and hence $P$-almost surely.
Consequently, we have $P$-almost surely
\begin{align*}
& \sup_{t \in [0,T]} \left\lvert \lbrack X(\omega),X(\omega) \rbrack^{\pi_n(\omega)}_t - [X,X]^P_t(\omega) \right\rvert  = \sup_{t \in [0,T]} \left\lvert (K^n \bullet X)_t(\omega) \right\rvert  \longrightarrow 0,
\end{align*}  
as $n \to \infty$.
This completes the proof for $X \in \mathcal{H}^4(P)$.
The general case is proved by the above argument and the usual prelocalization procedure.
\end{proof}

\begin{Cor} \label{5.1gb}
Suppose that, for all $P \in \mathcal{P}$, a {\cadlag} adapted process $X$ is a $P$-semimartingale
and $\Pi = (\pi_n)$ is a sequence of optional partitions
such that $\lvert \pi_n(\omega) \rvert \to 0$ holds $P$-almost surely.
Moreover, suppose that
for every $P \in \mathcal{P}$ and every $T \in \mathbb{R}_{\geq 0}$, 
the series $\sum_{n} O_T(X(\omega),\pi_n(\omega))$ converges $P$-almost surely.
Then, for all $P \in \mathcal{P}$,
$P$-almost every path of $X$ has quadratic variation along $\Pi$
that coincides with the usual quadratic variation $[X,X]^{P}$, $P$-almost surely.
\end{Cor}

\begin{Rem} \label{5.1h}
An example of the optional partitions used in Proposition~\ref{5.1f}
can be given in the following manner. For a {\cadlag} semimartingale $X$,
we define a sequence of stopping times inductively as
\begin{align*}
T_0^n(\omega) & = 0,   \\
T_{k+1}^{n}(\omega) & = \inf\left\{ t > T_k^{n}(\omega) \biggm\vert \lvert X_t(\omega) - X_{T_k^n}(\omega) \rvert > \frac{1}{2^{n+1}} \right\} \wedge \left( T_k^n(\omega) + \frac{1}{n} \right).
\end{align*}
Let $\pi_n = (T_k^n)_{k \in \mathbb{N}}$ and $\Pi = (\pi_n)_{n \in \mathbb{N}}$.
Then for each $\omega \in \Omega$, we have $\lvert T^n_k(\omega) - T^n_{k+1}(\omega) \rvert \leq 1/n$.
Furthermore
\begin{equation*}
\lvert X_s(\omega) - X_{u}(\omega) \rvert \leq \lvert X_s(\omega) - X_{T_{k}^n}(\omega) \rvert + \lvert X_u(\omega) - X_{T_{k}^n}(\omega) \rvert  \leq \frac{1}{2^{n}}
\end{equation*}
holds for arbitrary $s,u \in [T_k^n(\omega),T_{k+1}^{n}(\omega)[$.
Therefore, for all $P \in \mathcal{P}$,
\begin{equation*}
\sum_{n \in \mathbb{N}} O_t(X(\omega),\pi_n) \leq \sum_{n \in \mathbb{N}} \frac{1}{2^n} < \infty, \quad P\text{-a.s.}
\end{equation*}
\end{Rem}

\begin{Rem} \label{5.1i} 
Because of Theorem~\ref{2.2c} and Proposition~\ref{5.1f},
the It{\^o}-F{\"o}llmer integral $\int_0^t g(X_{s-}) dX_s$ for $g$ of $C^1$-class
converges almost surely along $\Pi = (\pi_n)$ provided that $\sum_{n} O_t(X,\pi_n) < \infty$ a.s..
Under this assumption, $\Pi$ seems to control the oscillations of the paths of $X$, only.
It is known that, in general, we have to control the oscillation of integrands 
to obtain the pathwise convergence of stochastic integrals:
see
%Karandikar~
\cite{Karandikar_1995}.
The oscillation of the paths of $g \circ X_{-}$ is controlled by that of $X$ in this particular case.
Hence, our result is valid.
\end{Rem}

\section{Integral equations with respect to It{\^o}-F{\"o}llmer integrals and their applications to finance}

We have seen that some fundamental formulas in stochastic integration theory are valid 
within the framework of It{\^o}-F{\"o}llmer integrals.
In this section, using these formulas, we compute explicit solutions to certain integral equations.

\subsection{Linear equations}

For a given $X \in QV(\Pi)$ and $H \in D(\mathbb{R}_{\geq 0},\mathbb{R})$,
we consider the following equation for $Y$.
\begin{equation} \label{3.1a}
Y_t = H_t + \int_0^t Y_{s-} dX_s .
\end{equation}
If $H$ is (not) constant, the equation \eqref{3.1a} is called an (in)homogeneous linear equation.
A solution of \eqref{3.1a} is an admissible integrand of $X$ that satisfies \eqref{3.1a}.

We first consider the following proposition.

\begin{Prop} \label{3.1c}
Let $X \in QV(\Pi)$ and let
\begin{equation*}
\mathcal{E}(X)_t = \exp\left( X_t -X_0 - \frac{1}{2}[X,X]^{\mathrm{c}}_t \right) \prod_{0 < s \leq t} (1 + \Delta X_s) e^{-\Delta X_s}.
\end{equation*}
Then $\mathcal{E}(X)$ is the unique solution to the homogeneous linear equation
\begin{equation} \label{3.1d}
Y_t = 1 + \int_0^t Y_{s-} dX_s.
\end{equation}
\end{Prop}

\begin{proof}
Let
\begin{equation*}
V(X)_t = \prod_{0 < s \leq t} (1+\Delta X_s) e^{-\Delta X_s}.
\end{equation*}
It is well known that this infinite product converges absolutely
and $V(X)$ is a {\cadlag} purely discontinuous function of locally finite variation.
(See, for example,
%Jacod and Shiryaev~
\cite[Chapter I, 4.61 Theorem]{Jacod_Shiryaev_1987}).
Here we define $f(x,a,b) = e^{x - X_0-a/2}b$.
It is easy to see that $f$ is of $C^{2}$-class.
Since
$\frac{\partial}{\partial x} f(X_t,[X,X]^{\mathrm{c}}_t,V_t) = \mathcal{E}(X)_t$ holds,
the function $\mathcal{E}(X)$ is an admissible integrand of $X$.
Applying the It{\^o}-F{\"o}llmer formula to $f(X_t,[X,X]^{\mathrm{c}}_t,V_t)$, we see that $Y = \mathcal{E}(X)$ satisfies \eqref{3.1d}.

For uniqueness, we will take an arbitrary solution $Y$ of \eqref{3.1d} and show that $Y=\mathcal{E}(X)$.
Let $Y$ be a solution of \eqref{3.1d},
and let $Z = X - X_0 -(1/2)[X,X]^{\mathrm{c}}$. 
Then we see that $\mathbf{X} := (Y,Z)$ belongs to $QV(\Pi;\mathbb{R}^2)$
by the representation
$Y_t = 1 + \int_0^t Y_{s-} dX_s$
and by Corollary~\ref{2.1z} and Proposition~\ref{2.3g}.
The It{\^o}-F{\"o}llmer formula implies the existence of the It{\^o}-F{\"o}llmer integral
$\int_0^t \left\langle \nabla g(Y_{s-},Z_{s-}),d\mathbf{X}_s \right\rangle$
where $g(y,z) = y e^{-z}$.
Since $Y$ is represented as $Y=F(A,X)$ by some $F$ of $C^{1,2}$-class and $A$ of locally finite variation,
the It{\^o}-F{\"o}llmer integral
$\int_0^t g(Y_{s-},Z_{s-}) dX_s$
exists. Moreover, by Remark~\ref{2.3d}
\begin{equation*}
\int_0^t g(Y_{s-},Z_{s-}) dZ_s = \int_0^t g(Y_{s-},Z_{s-}) dX_s - \frac{1}{2} \int_0^t g(Y_{s-},Z_{s-}) d[X,X]^{\mathrm{c}}_s
\end{equation*}
also exists.
Therefore, we can write
\begin{equation} \label{3.1e}
\int_0^t \left\langle \nabla g(Y_{s-},Z_{s-}),d\mathbf{X}_s \right\rangle =  -\int_0^t g(Y_{s-},Z_{s-}) dZ_s + \int_0^t e^{-Z_{s-}} dY_s.
\end{equation}
Applying the It{\^o}-F{\"o}llmer formula to $W = g(Y,Z) = Y e^{-Z}$ and using \eqref{3.1e},
we obtain
\begin{align*}
W_t - 1 & = \int_0^t -W_{s-} dZ_s + \int_0^t e^{-Z_{s-}} dY_s + \frac{1}{2} \int_{0}^{t} W_{s-} d[Z,Z]^{\mathrm{c}}_s  \\
& \quad + \int_{0}^{t} -e^{-Z_{s-}} d[Y,Z]^{\mathrm{c}}_t + \sum_{0 < s \leq t} \left\{ \Delta W_s + W_{s-} \Delta Z_s - e^{-Z_{s-}} \Delta Y_s \right\}  \\
& = \sum_{0 < s \leq t} \Delta W_s.
\end{align*}
Note that $W$ satisfies
\begin{equation} \label{3.1f}
W_t = 1+ \int_0^t W_{s-} dR_s,
\end{equation}
where $R$ is a purely discontinuous function of locally finite variation defined by
\begin{equation*}
R_t = \sum_{0 < s \leq t} \left\{ (1+\Delta X_s)e^{-\Delta X_s} - 1 \right\}.
\end{equation*}
The function $U := \mathcal{E}(X) e^{-Z}$ clearly satisfies \eqref{3.1f}
because $\mathcal{E}(X)$ is a solution of \eqref{3.1d}.
Hence 
\begin{equation*}
W_t -U_t = \int_0^t (W_{s-} - U_{s-}) dR_s
\end{equation*}
holds for all $t$.
Gronwall's lemma for Stieltjes integrals implies $W = U$,
and so $Y = \mathcal{E}(X)$.
This establishes uniqueness. 
\end{proof}

\begin{Rem} \label{3.1g}
If $X \in QV(\Pi)$ satisfies $\Delta X_t \neq -1$ for all $t \geq 0$, we have 
$\mathcal{E}(X)_t \neq 0$ and $\mathcal{E}(X)_{t-} \neq 0$ for all $t \geq 0$.
Moreover, if $\Delta X_t > -1$ holds for all $t \geq 0$,
then $\mathcal{E}(X)_t > 0$ and $\mathcal{E}(X)_{t-} > 0$ hold for all $t \geq 0$.
\end{Rem}

Remarks \ref{2.2fa} and \ref{3.1g} allow us to apply the It{\^o}-F{\"o}llmer formula
(Theorem~\ref{2.2c}) to $1/\mathcal{E}(X)_t$.
Consequently, we obtain the following lemma.

\begin{Lem} \label{3.1ga}
Let $X \in QV(\Pi)$ satisfy $\Delta X_{t} \neq -1$ for all $t \in \mathbb{R}_{\geq 0}$.
Then $1/\mathcal{E}(X) \in QV(\Pi)$ and the following representation holds.
\begin{align}  \label{3.1gb}
\frac{1}{\mathcal{E}(X)_t} - 1
& = - \int_0^t \frac{dX_s}{\mathcal{E}(X)_{s-}} + \int_{0}^{t} \frac{d[X,X]^{\mathrm{c}}_s}{\mathcal{E}(X)_{s-}} + \sum_{0 < s \leq t} \frac{1}{\mathcal{E}(X)_{s-}} \left[ \frac{(\Delta X_s)^2}{1 + \Delta X_s} \right].
\end{align}
\end{Lem}

With the help of Proposition~\ref{3.1c} and Lemma~\ref{3.1ga}, we obtain the following proposition.
The first expression \eqref{3.1i} was introduced by
%Jaschke~
\cite{Jaschke_2003}
for inhomogeneous linear SDEs.

\begin{Prop} \label{3.1h}
Suppose that $X \in QV(\Pi)$ satisfies
$\Delta X_s \neq -1$ for all $t \in \mathbb{R}_{\geq 0}$
and that $H$ is an admissible integrand of $X$.
Define
\begin{equation} \label{3.1i}
Z_t = H_t - \mathcal{E}(X)_t \int_0^t H_{s-} d\left( \frac{1}{\mathcal{E}(X)} \right)_s, \quad t \in \mathbb{R}_{\geq 0}.
\end{equation}
Then, it is the unique solution to the following inhomogeneous linear equation.
\begin{equation} \label{3.1j}
Z_t = H_t + \int_0^t Z_{s-} dX_s.
\end{equation}
Moreover, if $H$ is represented as
\begin{equation} \label{3.1ja}
H_t = \int_0^t \xi_{s-} dX_s + A_t
\end{equation}
by an admissible integrand $\xi$ and $A \in FV_{\mathrm{loc}}$, then $Z$ admits another expression as follows.
\begin{align} \label{3.1jb}
Z_t
&= \mathcal{E}(X)_t \left(  H_0 + \int_0^t \frac{dH_s}{\mathcal{E}(X)_{s-}} - \int_0^t \frac{d[H,X]^{\mathrm{c}}_s}{\mathcal{E}(X)_{s-}} - \sum_{0 < s \leq t} \frac{\Delta H_s \Delta X_s}{\mathcal{E}(X)_{s-} \left(1+\Delta X_s \right)} \right).
\end{align}
\end{Prop}

\begin{proof}
We first prove that $Z$ is a solution of \eqref{3.1j}.
The existence of the It{\^o}-F{\"o}llmer integral $Y_t := \int_0^t H_{s-} d(\mathcal{E}(X)^{-1})_s$
follows from the assumption about $H$, Theorem~\ref{2.3m}, and Lemma~\ref{3.1ga}.
Applying Corollary~\ref{2.3s} and Proposition~\ref{2.3g} to $Y \mathcal{E}(X)$,
we get
\begin{align*}
Z_t - H_t 
& =  \int_0^t H_{s-} dX_s - \int_0^t H_{s-} d[X,X]^{\mathrm{c}}_s  - \sum_{0 < s \leq t} \frac{H_{s-}(\Delta X_s)^2}{1+ \Delta X_s}  \notag \\
& \quad - \int_0^t Y_{s-} \mathcal{E}(X)_{s-} dX_s + \int_0^t H_{s-} d[X,X]_s - \sum_{0 < s \leq t} \frac{H_{s-} (\Delta X_s)^3}{1 + \Delta X_s}  \notag \\
& = \int_0^t Z_{s-} dX_s.
\end{align*}
Hence $Z$ satisfies equation \eqref{3.1j}.

Next we show uniqueness.
Let $W$ be an arbitrary solution of \eqref{3.1j}.
Corollary~\ref{2.3s} yields
\begin{align*}
\frac{W_t-H_t}{\mathcal{E}(X)_t}
& = \int_0^t (W_{s-}- H_{s-}) d\left( \frac{1}{\mathcal{E}(X)} \right)_s + \int_0^t \frac{W_{s-}}{\mathcal{E}(X)_{s-}} dX_s   + \left[ W-H,\frac{1}{\mathcal{E}(X)} \right]_t.
\end{align*}
Using Proposition~\ref{2.3g} and Corollary~\ref{2.1z}, the quadratic covariation part is given by
\begin{align*}
\left[ W-H,\frac{1}{\mathcal{E}(X)} \right]_t 
=  - \int_0^t \frac{W_{s-}}{\mathcal{E}(X)_{s-}} d[X,X]^{\mathrm{c}}_s - \sum_{0 < s \leq t} \frac{W_{s-}}{\mathcal{E}(X)_{s-}} \frac{(\Delta X_s)^2}{1 + \Delta X_s}.
\end{align*}
Therefore,
\begin{equation*}
\frac{W_t-H_t}{\mathcal{E}(X)_t} = - \int_0^t H_{s-} d\left( \frac{1}{\mathcal{E}(X)} \right)_s.
\end{equation*}
This indicates $W = Z$ and establishes the uniqueness of solutions.

Now, suppose $H$ is represented as \eqref{3.1ja}.
By Corollary~\ref{2.3s}, Proposition~\ref{2.3g}, and Lemma~\ref{3.1ga} we have
\begin{align*}
\frac{H_t}{\mathcal{E}(X)_t} - H_0 
& = \int_0^t \frac{dH_s }{\mathcal{E}(X)_{s-}} + \int_0^t H_{s-} d\left( \frac{1}{\mathcal{E}(X)} \right)_{s} - \int_0^t \frac{d\left[H,X \right]_s}{\mathcal{E}(X)_{s-}} - \sum_{0 < s \leq t} \frac{1}{\mathcal{E}(X)_{s-}}\frac{\Delta H_s \Delta X_s}{1 + \Delta X_s}.
\end{align*}
Combining this and \eqref{3.1i}, we obtain the expression of \eqref{3.1jb}.
\end{proof}

\subsection{Nonlinear equations}

By combining Proposition~\ref{3.1h} and results in
%Duan and Yan~
\cite[Section 3]{Duan_Yan_2008}, we can solve a certain class of nonlinear equations.

\begin{Prop} \label{3.1la}
Let $X \in QV(\Pi)$ satisfy $\Delta X_t \neq -1$ for all $t$,
and let $f$ be a continuous function on $\mathbb{R}_{\geq 0} \times \mathbb{R}$.
Suppose that the integral equation
\begin{equation} \label{3.1n}
Y_t = x + \int_0^t \frac{f(s,Y_s \mathcal{E}(X)_s)}{\mathcal{E}(X)_s} ds
\end{equation}
with respect $Y$ has a unique solution.
Then, the path $Z = Y \mathcal{E}(X)$, where $Y$ denotes the unique solution of \eqref{3.1n}, is the unique solution of 
\begin{equation} \label{3.1m}
Z_t = x + \int_0^t f(s,Z_s) ds + \int_0^t Z_{s-} dX_s.
\end{equation}
\end{Prop}

Here, the terminology ``solution of \eqref{3.1m}'' means
an admissible integrand of $X$ that satisfies \eqref{3.1m}.

\begin{proof}
We see that, by Corollary~\ref{2.3s},
\begin{align*}
Y_t \mathcal{E}(X)_t
& = \int_0^t \mathcal{E}(X)_{s} \frac{f(s,Y_s \mathcal{E}(X)_s)}{\mathcal{E}(X)_s} ds + \int_0^t Y_{s} \mathcal{E}(X)_{s-} dX_s  \notag \\
& = \int_0^t f(s,Y_s \mathcal{E}(X)_s ) ds + \int_0^t Y_{s-} \mathcal{E}(X)_{s-} dX_s.
\end{align*}
Therefore, $Y \mathcal{E}(X)$ is a solution of \eqref{3.1m}.
%Note here that $Y$ is continuous.
For uniqueness, let $Z$ be an arbitrary solution of \eqref{3.1m} and let
\begin{equation*}
H_t = x + \int_0^t f(s,Z_s) ds.
\end{equation*}
Then, by Proposition~\ref{3.1h}, $Z$ satisfies
\begin{equation*}
Z_t = \mathcal{E}(X)_t \left( x + \int_0^t \frac{f(s,Z_s)}{\mathcal{E}(X)_{s}} ds \right).
\end{equation*}
Hence the path $W := Z/\mathcal{E}(X)$ satisfies \eqref{3.1n}.
By the uniqueness of solutions of \eqref{3.1n}, we have
$Z = W \mathcal{E}(X) = Y \mathcal{E}(X)$.
\end{proof}

\begin{Rem}
Equation \eqref{3.1n} has a unique solution if, for example, $f$ satisfies the following conditions.
\begin{itemize}
\item (Local Lipschitz condition) For each $T, M > 0$, there is a constant $L_{T,M} > 0$ such that
for all $x,y \in [-M,M]$ and all $t \in [0,T]$,
\begin{equation*}
\lvert f(t,x) - f(t,y) \rvert \leq L_{T,M} \lvert x-y \rvert.
\end{equation*}
\item (Linear growth condition) For each $T > 0$, there is a constant $K_T > 0$ such that 
for all $(t,x) \in [0,T] \times \mathbb{R}$,
\begin{equation*}
\lvert f(t,x) \rvert \leq K_T (1 + \lvert x \rvert).
\end{equation*}
\end{itemize}
Under these assumptions,
$g(t,y) = f(t, y \,\mathcal{E}(X)_t )/\mathcal{E}(X)_t$
again satisfies both local Lipschitz and linear growth conditions.
Indeed, fix $T > 0$ and $M > 0$ and define $M' = \sup_{t \in [0,T]} \lvert \mathcal{E}(X)_t \rvert M$.
Then $g$ satisfies
\begin{equation*}
\lvert g(t,x) - g(t,y) \rvert
= L_{T,M'} \lvert y-x \rvert.
\end{equation*}
Moreover, it satisfies
\begin{equation*}
\lvert g(t,y) \rvert
\leq K_T \left( \frac{1}{\inf_{t \in [0,T]} \lvert \mathcal{E}(X)_t \rvert} + 1 \right)\left( 1 + \lvert y \rvert \right).
\end{equation*}
By the assumption that $\Delta X_t \neq -1$, Remark~\ref{3.1g} and the {\cadlag} property of $\mathcal{E}(X)$, 
we see that $\inf_{[0,T]} \mathcal{E}(X)_{t} > 0$.
Hence, $g$ satisfies the linear growth condition.
\end{Rem}

\begin{Rem} \label{3.1o}
There is another class of nonlinear equations that is solvable within our framework.
%Mishura and Schied~
\cite{Mishura_Schied_2016} discuss an extension of the so called Doss-Sussmann method within the framework of F{\"o}llmer's calculus for continuous paths.
The Doss-Sussmann method is a way to solve SDEs by using the solution of ODEs~\cite{Doss_1977,Sussmann_1978}.
\end{Rem}

\subsection{Drawdown equations}

In this section, we deal with integral equations which are called
drawdown equations.
We follow
%Carraro, El Karoui, and Ob{\l}{\'{o}}j~
\cite[Section 2--3]{Carraro_ElKaroui_Obloj_2012},
which studies Az{\'e}ma-Yor processes and related drawdown equations.
We interpret their results in our pathwise setting.

Given an $\mathbb{R}$-valued {\cadlag} path $X$,
we define its running maximum, denoted by $\overline{X}$, through the formula
$\overline{X}_t = \sup_{0 \leq s \leq t} X_s$ for all $t \in \mathbb{R}_{\geq 0}$.
Then the path $\overline{X}$ is {\cadlag} and increasing.
In this subsection, we always assume that $\overline{X}$ is continuous.
Recall that this continuity assumption implies
\begin{equation} \label{3.3a}
\int_{0}^{t} (\overline{X}_s - X_s) d\overline{X}_s = 0, \quad t \in \mathbb{R}_{\geq 0}.
\end{equation}

\begin{Prop} \label{3.3b}
Suppose that $X \in QV(\Pi)$ satisfies $X_0 = a$ and that the running maximum $\overline{X}$ is continuous.
Let $U: [a,\infty\mathclose{[} \to \mathbb{R}$ be a $C^2$ function such that $U(a) = a^*$.
(We consider only right derivatives at $a$.)
Then the function
\begin{equation} \label{3.3c}
M^U_t(X) := U(\overline{X}_t) - U'(\overline{X}_t)(\overline{X}_t -X_t) 
\end{equation}
has quadratic variation along $\Pi$ and satisfies 
\begin{equation} \label{3.3d}
M^U_t(X) = a^* + \int_0^t U'(\overline{X}_s) dX_s.
\end{equation}
\end{Prop}

\begin{proof}
The function $f(a,x) = U(a) - U'(a)(a-x)$ is clearly belongs to $C^{1,2}$-class.
The existence of the quadratic variation of $M^U(X) = f(\overline{X},X)$ follows from 
Proposition~\ref{2.1m}.
Applying the It{\^o}-F{\"o}llmer formula to $f(\overline{X},X)$, we obtain 
\begin{align*}
f(\overline{X}_t,X_t) - a^*
= - \int_{0}^{t} U''(\overline{X}_s) (\overline{X}_s - X_s)d\overline{X}_s + \int_0^t U'(\overline{X}_s) dX_s.
\end{align*}
By \eqref{3.3a} we see that 
\begin{equation*}
\int_{0}^{t} \left\lvert U''(\overline{X}_s) (\overline{X}_s - X_s) \right\rvert d\overline{X}_s \leq \sup_{s \in [0,t]}\lvert U''(\overline{X}_s) \rvert  \int_{0}^{t} (\overline{X}_s - X_s)d\overline{X}_s = 0.
\end{equation*}
Hence
\begin{equation*}
M^U_t(X) = f(\overline{X}_t,X_t) = a^* + \int_0^t U'(\overline{X}_s) dX_s.
\end{equation*}
holds for all $t \in \mathbb{R}_{\geq 0}$.
\end{proof}

\begin{Def} \label{3.3e}
We call the function $M^U(X)$ defined in Proposition~\ref{3.3b}
\emph{the Az{\'e}ma-Yor path associated with $U$ and $X$}.
\end{Def}

The following proposition, which was originally proved by
%Carraro, El Karoui, and Ob{\l}{\'{o}}j~
\cite{Carraro_ElKaroui_Obloj_2012},
is still valid within the framework of It{\^o}-F{\"o}llmer integration.
It is easy to see that the proof of \cite[Proposition 2.2]{Carraro_ElKaroui_Obloj_2012}
is pathwise and does not use It{\^o} calculus.
So we can replace the word ``max-continuous semimartingale'' with ``path of $QV(\Pi)$ with continuous running maximum.''

\begin{Prop}[ {\cite[Proposition 2.2]{Carraro_ElKaroui_Obloj_2012}} ] \label{3.3f}
%Carraro, El Karoui, and Ob{\l}{\'{o}}j~
\begin{enumerate}
\item Let $X$, $U$ satisfy the same assumptions as in Proposition~\ref{3.3b}. Moreover, we suppose that $U$ is increasing. Then $M^U(X)$ still has a continuous running maximum and satisfies
$\overline{M^U_t(X)} = U(\overline{X}_t)$ for all $t \in \mathbb{R}_{\geq 0}$.
\item In addition to the assumptions of (i), let $F$ be an increasing $C^{2}$ function such that $U \circ F$ is well defined. Then,
$M^U_t(M^F(X)) = M_t^{U \circ F}(X)$.
\end{enumerate}
\end{Prop}

The following is a direct consequence of Propositions \ref{3.3f} and \ref{3.3b}.

\begin{Cor}[ {\cite[Corollary 2.4]{Carraro_ElKaroui_Obloj_2012}} ] \label{3.3g}
%Carraro, El Karoui, and Ob{\l}{\'{o}}j~
Consider a strictly increasing $C^2$ function $U$ defined on $[a,\infty[$ such that $U(a) = a^*$.
Let $V: [a^*,U(\infty)\mathclose{[} \to [a,\infty[$ denotes the inverse of $U$
where $U(\infty) = \lim_{x \to \infty} U(x)$.
Moreover, suppose that $X \in QV(\Pi)$ has continuous running maximum and satisfies $X_0 = a$.
Then, we have $X_{t} = M^V_{t}(M^U(X))$.
If we define $Y_t = M^U_t(X)$, $X$ and $Y$ are expressed as follows.
\begin{equation*}
Y_t = a^* + \int_0^t U'(\overline{X}_s)dX_s,\quad X_t = a + \int_0^t V'(\overline{Y}_t) dY_t, \quad t \in \mathbb{R}_{\geq 0}.
\end{equation*}
\end{Cor}

\begin{Def} \label{3.3h}
Let $X$ be a real-valued {\cadlag} path and 
$w$ be a real-valued function defined on some subset of
$\mathbb{R}$ including $[X_0,\infty[$.
We say that \emph{$X$ satisfies the $w$-drawdown constraint} if 
$X_t \wedge X_{t-} > w(\overline{X}_t)$
holds for all $t \in \mathbb{R}_{\geq 0}$.
\end{Def}

Let us consider the Az{\'e}ma-Yor path $M^U(X)$ associated
with a strictly positive valued $X$ and a strictly increasing $U$.
We can construct a function $w$ such that $M^U(X)$ satisfies the $w$-drawdown constraint.
Indeed, let $V = U^{-1}$, $h(x) = U(x) - xU' (x)$, and $w= h \circ V$.
Then we have
\begin{align*}
& M^U_t(X) \wedge M^U_{t-}(X) = U(\overline{X}_t) - U'(\overline{X}_t)\overline{X}_t + U'(\overline{X}_t) (X_t \wedge X_{t-})  \\
& > U(\overline{X}_t) - U'(\overline{X}_t)\overline{X}_t = h(V(U(\overline{X}_t))) = w(\overline{M^U_t(X)}).
\end{align*}
Conversely, given a ``nice'' function $w$, we can construct a {\cadlag} path that satisfies
the $w$-drawdown constraint as follows.

\begin{Prop} \label{3.3i}
\begin{enumerate}
\item Let $X \in QV(\Pi)$ satisfy $X_0 = a$ and $X_t \wedge X_{t-} > 0$.
Suppose that its running maximum is continuous. 
We consider a function $w \colon [a^*,\infty\mathclose{[} \to \mathbb{R}$ of $C^1$-class
satisfying $y - w(y) > 0$ for all $y \in [a^*,\infty[$.
Moreover, we define
\begin{equation*}
V(y) = a \exp\left(\int_{[a^*,y]} \frac{1}{s-w(s)} ds \right), \quad y \in [a^*,\infty[
\end{equation*}
and $U = V^{-1}$.
Then $M^U(X)$ is the unique solution of
\begin{equation} \label{3.3j}
Y_t = a^* + \int_0^t \frac{Y_{s-} - w(\overline{Y}_s)}{X_{s-}} dX_s
\end{equation}
that satisfies the $w$-drawdown constraint and has continuous running maximum.

\item
Let $Y \in QV(\Pi)$. 
Suppose that $Y$ satisfies the $w$-drawdown constraint, $Y_0 = a^*$, and its running maximum is continuous.
Then, $X_t = M_t^V(Y)$ is the unique {\cadlag} path
of $QV(\Pi)$ such that $\overline{X}$ is continuous, $X_0 = a$ and
$X$ satisfies \eqref{3.3j}.
\end{enumerate}
\end{Prop}

The integral equation \eqref{3.3j} is called a drawdown equation.
The terminology ``solution of \eqref{3.3j}'' means an admissible integrand of $X$ that satisfies equation \eqref{3.3j}.
To prove uniqueness, we use the following lemma.

\begin{Lem} \label{3.3k}
Suppose $X$ and $Y$ of $QV(\Pi)$ satisfy $X_t$, $X_{t-}$, $Y_t$, $Y_{t-} > 0$, $X_0 = Y_0$ and
\begin{equation*}
\int_{0}^{t} \frac{dX_s}{X_{s-}} = \int_{0}^{t} \frac{dY_s}{Y_{s-}}
\end{equation*}
for all $t \in \mathbb{R}_{\geq 0}$. Then $X=Y$.
\end{Lem}

\begin{proof}
Let 
\begin{equation*}
Z_t = \int_{0}^{t} \frac{dX_s}{X_{s-}} = \int_{0}^{t} \frac{dY_s}{Y_{s-}}, \quad t \in \mathbb{R}_{\geq 0}.
\end{equation*}
Then we can easily check that both $X$ and $Y$ are admissible integrands of $Z$.
Theorem~\ref{2.3m} implies that $X/X_0$ and $Y/Y_0$ are
solutions to the homogeneous linear equation driven by $Z$ with initial value $1$.
Hence, by Proposition~\ref{3.1c},
we obtain $X = X_0 \mathcal{E}(Z) = Y_0 \mathcal{E}(Z) = Y$.
\end{proof}

\newtheorem{proof3.3i}{Proof of Proposition~\ref{3.3i}}

\begin{proof3.3i}
(i)
We first prove that $Y=M^U(X)$ satisfies \eqref{3.3j}.
By definition, the derivative of $V$ is $V'(y) = V(y)/ (y-w(y))$ and
the derivative of $U$ is $U' = 1/V'$.
Therefore, by Proposition~\ref{3.3f} and Corollary~\ref{3.3g},
we have
\begin{equation*}
Y_{t-} - w(\overline{Y}_t) = Y_{t-} -  U(\overline{X}_t) + U'(\overline{X}_t)V(U(\overline{X}_t)) = U'(\overline{X}_t)X_{t-}.
\end{equation*}
This proves
\begin{equation*}
Y_t = \int_0^t U'(\overline{X}_s)dX_s = \int_0^t \frac{Y_{s-}-w(\overline{Y}_s)}{X_{s-}} dX_s.
\end{equation*}
The It{\^o}-F{\"o}llmer integral of this equation is well-defined since $U'(\overline{X})$
has locally finite variation.

Next, we show uniqueness.
Suppose that $Y$ is a solution of \eqref{3.3j} satisfying the $w$-drawdown constraint
and that the running maximum $\overline{Y}$ is continuous. 
Since $Y$ satisfies the $w$-drawdown constraint, $M_t^V(Y)$ and $M^V_{t-}(Y)$
are strictly positive. 
It is easy to see that $\frac{Y - w(\overline{Y})}{X}$ is an admissible integrand of $X$,
and $\frac{1}{Y - w(\overline{Y})}$ is an admissible integrand of $X$.
Then we have 
\begin{equation*}
\int_0^t \frac{dY_s}{Y_{s-}-w(\overline{Y}_s)} = \int_0^t \frac{1}{Y_{s-}-w(\overline{Y}_s)} \frac{Y_{s-} - w(\overline{Y}_s)}{X_{s-}} dX_s = \int_0^t \frac{dX_s}{X_{s-}}
\end{equation*}
thanks to Theorem~\ref{2.3m}.
Moreover, we see that
\begin{equation*}
\int_0^t \frac{dY_s}{Y_{s-}-w(\overline{Y}_s)}
= \int_0^t \frac{V'(\overline{Y}_s)}{M^V_{s-}(Y)} dY_s =  \int_0^t \frac{dM^V_s(Y)}{M^V_{s-}(Y)}.
\end{equation*}
The second equality follows from Theorem~\ref{2.3m}.
Hence,
\begin{equation*}
\int_0^t \frac{dX_s}{X_{s-}} = \int_0^t \frac{dM^V_s(Y)}{M^V_{s-}(Y)}.
\end{equation*}
This proves $X = M^V(Y)$ because of Lemma~\ref{3.3k}.
Then we can deduce that $Y_t = M^U_t(X)$ from Corollary~\ref{3.3g}.

(ii)
Suppose that $Y \in QV(\Pi)$ satisfies the $w$-drawdown constraint and $Y_0 = a^*$,
and that the running maximum $\overline{Y}$ is continuous.
Let $X := M^V(Y)$.
Then, $Y$ is an admissible integrand of $X$ due to Corollary~\ref{3.3g} and the definition of an Az{\'e}ma-Yor path.
Hence $\frac{Y - w(\overline{Y})}{M^V(Y)}$ is also an admissible integrand of $X$.
By the associativity of It{\^o}-F{\"o}llmer integrals, we have
\begin{equation*}
\int_0^t \frac{Y_{s-}-w(\overline{Y}_s)}{M^V_{s-}(Y)} dM^V(Y) = \int_0^t \frac{Y_{s-}-w(\overline{Y}_s)}{M^V_{s-}(Y)} V'(\overline{Y}_s) dY_s = \int_0^t \frac{M^V_{s-}(Y)}{M^V_{s-}(Y)} dY_s = Y_t - Y_0.
\end{equation*}
Therefore $X = M^V(Y)$ satisfies \eqref{3.3j}.
Let $Z$ be any {\cadlag} path of $QV(\Pi)$ satisfying \eqref{3.3j} such that $Z_0 = a$
and $\overline{Z}$ is continuous.
Then we can show
\begin{equation*}
\int_0^t \frac{dZ_s}{Z_{s-}} = \int_0^t \frac{dM^V_s(Y)}{M^V_{s-}(Y)},
\end{equation*}
in the same way as the proof of the first part.
Consequently, we obtain $M^V(Y) = Z$.
\end{proof3.3i}

\subsection{Applications to Finance}

\subsubsection{Model-free CPPI with jumps}

As an application of the results for linear integral equations,
we consider model-free constant proportion portfolio insurance (CPPI) and dynamic proportion portfolio insurance (DPPI) strategies with jumps,
which extend the results of
%Schied~
\cite{Schied_2014}.
Throughout this section, we will fix a sequence of partitions $\Pi = (\pi_n)_{n \in \mathbb{N}}$
such that $\lvert \pi_n \rvert \to 0$ as $n \to \infty$.

Assume that there is one riskless asset and one risky asset in the market.
The price process of the riskless asset is denoted by $B$. Mathematically this is a path of locally finite variation with $B_0 = 1$.
The risky asset $S$ belongs to $QV(\Pi)$.
We suppose that both $B$ and $S$ are strictly positive valued 
and that $B_{t-}$ and $S_{t-}$ are also strictly positive.

We consider a portfolio that consists of $\xi_t$ units of the risky asset and $\eta_t$ units of the riskless asset at time $t$.
We will assume that $\xi$ is an admissible integrand of $S$ and $\eta$ is a {\cadlag} path.
The pair $(\xi,\eta)$ is called a trading strategy.
The value of the portfolio at time $t$ is defined by 
$V_t = \xi_t S_t + \eta_t B_t$.
A strategy $(\xi,\eta)$ is self-financing if
\begin{equation*}
V_t = V_0 + \int_0^t \xi_{s-} dS_s + \int_0^t \eta_{s-} dB_s
\end{equation*}
holds for all $t$.

Let $K$ be a nonnegative {\cadlag} path of locally finite variation.
We will construct a self-financing trading strategy $(\xi,\eta)$
such that $V_t \geq K_t$ for all $t$.
%Schied~
\cite{Schied_2014}
discusses a CPPI strategy
and its dynamic extension, a DPPI strategy,
when the floor constraint is represented in the form $K_t = \alpha V_0 B_t$ and asset prices are continuous.
We will relax the assumptions of these results.

A DPPI strategy is a trading strategy given by 
\begin{equation} \label{3.4b}
\xi_t = \frac{m_t (V_t - K_t)}{S_t}, \quad \eta = \frac{V_t - \xi_t S_t}{B_t},
\end{equation}
where the multiplier $m \colon \mathbb{R}_{\geq 0} \to \mathbb{R}$ is a {\cadlag} path.
Moreover, we assume $m$ to be an admissible integrand of $X$.
When $m$ is constant, this is called a CPPI strategy.

We will first discuss the existence of a self-financing DPPI strategy.
If a DPPI strategy $(\xi,\eta)$ with multiplier $m$ is self-financing,
the value process must satisfy
\begin{equation} \label{6.1c}
V_t - V_0 = \int_0^t \frac{m_{s-} (V_{s-} -K_{s-})}{S_{s-}} dS_s + \int_0^t \frac{V_{s-} - m_{s-} (V_{s-} - K_{s-})}{B_{s-}} dB_s.
\end{equation}
Set
\begin{equation*}
X_t = \int_0^t \frac{m_{s-}}{S_{s-}} dS_s + \int_0^t \frac{(1 - m_{s-})}{B_{s-}} dB_s , \quad 
H_t = V_0 - \int_0^t K_{s-} \left[ dX_s - \frac{1}{B_{s-}} dB_s \right]  .
\end{equation*}
These paths are well-defined by assumption.
Then \eqref{6.1c} can be rewritten as 
\begin{equation} \label{6.1f}
V_t = H_t + \int_0^t V_{s-} dX_s.
\end{equation}
This is an inhomogeneous linear integral equation.
To apply Proposition~\ref{3.1h} to \eqref{6.1f} we need to assume 
that $X$ satisfies $\Delta X_t \neq -1$ for all $t$.
Under this assumption
\begin{align} \label{6.1fa}
V_t
&= \mathcal{E}(X)_t \left(H_0 + \int_0^t \frac{dH_s }{\mathcal{E}(X)_{s-}}- \int_0^t \frac{d[H,X]^{\mathrm{c}}_s}{\mathcal{E}(X)_{s-}}   - \sum_{0 < s \leq t} \frac{1}{\mathcal{E}(X)_{s-}} \frac{\Delta H_s \Delta X_s}{1+ \Delta X_s} \right)
\end{align}
is the unique solution of \eqref{6.1f}.
Consequently, if $V$ is given by \eqref{6.1fa}, the DPPI strategy of \eqref{3.4b}
is self-financing.

We next consider the condition $V_t \geq K_t$.
By the definition of $H$, we see that $V_t /\mathcal{E}(X)_t$ has the representation 
\begin{align} \label{6.1h}
\frac{V_t}{\mathcal{E}(X)_t}
&= V_0 + \int_0^t \frac{K_{s-}}{\mathcal{E}(X)_{s-}} \left\{ -dX_s + \frac{1}{B_{s-}}dB_s + d[X,X]^{\mathrm{c}}_s \right\}    \notag \\
& \quad + \sum_{0 < s \leq t} \frac{K_{s-}}{\mathcal{E}(X)_{s-}} \cdot \frac{ (\Delta X_s)^2 - B_{s-}^{-1} \Delta B_s \Delta X_s}{1+ \Delta X_s}.
\end{align}
Using Lemma~\ref{3.1ga}, the integration by parts formula (Corollary~\ref{2.3s}) and Corollary~\ref{2.1z},
we have
\begin{align}  \label{6.1ha}
\frac{K_t}{\mathcal{E}(X)_t} - K_0
& = \int_0^t \frac{K_{s-}}{\mathcal{E}(X)_{s-}} \left( - dX_s + d[X,X]^{\mathrm{c}}_s + \frac{dK_s}{K_{s-}} \right)   \notag \\
& \quad + \sum_{0 < s \leq t} \frac{1}{\mathcal{E}(X)_{s-}} \frac{-\Delta K_s \Delta X_s + K_{s-} (\Delta X_s)^2}{1+\Delta X_s}.
\end{align}
Combining \eqref{6.1h} and \eqref{6.1ha}, we obtain 
\begin{align} \label{6.1i}
\frac{V_t}{\mathcal{E}(X)_t} - \frac{K_t}{\mathcal{E}(X)_t}
& = V_0 - K_0 + \int_0^t \frac{1}{\mathcal{E}(X)_{s-}} \left( \frac{K_{s-}}{B_{s-}} dB_s - dK_s \right)   \notag \\
& \quad + \sum_{0 < s \leq t} \frac{1}{\mathcal{E}(X)_{s-}} \cdot \frac{\Delta K_s \Delta X_s - K_{s-}B^{-1}_{s-} \Delta B_s \Delta X_s}{1+\Delta X_s}.
\end{align}

For general $K$, it is hard to know by \eqref{6.1i} whether $V$ satisfies $V_t \geq K_t$ or not.
So we concentrate our attention on the case where the floor constraint has the form $K = LB$ ($L \in FV_{\mathrm{loc}}$ and $L \geq 0$).
In this case, \eqref{6.1ha} can be transformed to the following form by applying the integration by parts formula to $K = LB$:
\begin{equation*}
\frac{V_t}{\mathcal{E}(X)_t} - \frac{K_t}{\mathcal{E}(X)_t}
= V_0 - L_0 - \int_0^t \frac{B_{s-}}{\mathcal{E}(X)_{s-}} dL^{\mathrm{c}}_s - \sum_{0 < s \leq t} \frac{B_s}{\mathcal{E}(X)_{s-}} \cdot\frac{ \Delta L_s}{1+\Delta X_s}.
\end{equation*}
Above formula implies that $V_t / \mathcal{E}(X)_t - K_t /\mathcal{E}(X)_t \geq 0$ holds for all $t$
if $L$ is nondecreasing and satisfies $V_0 \geq L_0$.
In addition, if $\mathcal{E}(X)_t > 0$, we have $V_t \geq K_t$.

These observations are summarized in the following proposition.
\begin{Prop} \label{6.1l}
Let $S \in QV(\Pi)$ be a price process of a risky asset, $B \in FV_{\mathrm{loc}}$ be that of a riskless asset and $V_0 \geq 0$ be the initial wealth.
Suppose that $S,S_{-},B,B_{-}$ are strictly positive and $L$ is a nonincreasing path such that $V_0 \geq L_0$.
For an admissible integrand $m: \mathbb{R}_{\geq 0} \to \mathbb{R}$ of $S$, we define a path $X$ by 
\begin{equation*}
X_t = \int_0^t \frac{m_{s-}}{S_{s-}} dS_s + \int_0^t \frac{(1 - m_{s-})}{B_{s-}} dB_s.
\end{equation*}
If $m$ satisfies $\Delta X_t \neq -1$ for all $t$,
the DPPI strategy given below is self-financing.
\begin{gather*}
V_t = L_t B_t + \mathcal{E}(X)_{t} \left( V_0 - L_0 - \int_0^t \frac{B_{s-}}{\mathcal{E}(X)_{s-}} dL^{\mathrm{c}}_s - \sum_{0 < s \leq t} \frac{B_s}{\mathcal{E}(X)_{s-}} \frac{ \Delta L_s}{1+\Delta X_s}  \right),  \notag \\
\xi_t = \frac{m_t (V_t - L_t B_t)}{S_t}, \quad \eta_t = \frac{V_t - \xi_t S_t}{B_t}.
\end{gather*}
Moreover, if the condition $\Delta X_{t} > -1$ holds for all $t \geq 0$, the value process of this portfolio 
satisfies $V_t \geq L_t B_t$ for all $t$.
\end{Prop}

Suppose that the multiple $m$ is constant in Proposition~\ref{6.1l}.
In this case, the trading strategy $(\xi,\eta)$ is called CPPI. 
If $m_t = \overline{m}$ for all $t$, the process $X$ in the proposition is
\begin{equation*}
X_t = \overline{m} \int_0^t \frac{1}{S_{s-}} dS_s + (1- \overline{m})\int_0^t \frac{1}{B_{s-}} dB_s.
\end{equation*}
Then we have
\begin{equation*}
\Delta X_t = \overline{m} \frac{\Delta S_t}{S_{t-}} + (1-\overline{m}) \frac{\Delta B_t}{B_{t-}}.
\end{equation*}
By the assumptions $S, B > 0$, the conditions $\Delta S_t/S_{t-} > -1$ and $\Delta B_t/B_{t-} > -1$ are always satisfied.
If we choose $\overline{m}$ such that $0 \leq \overline{m} \leq 1$, we have $\Delta X_{t-} > -1$.

\subsubsection{Portfolio strategies satisfying drawdown constraint}

We will consider portfolio strategies satisfying the drawdown constraint
as an application of observations about drawdown equations. 
Let $S,B$ be the same paths as in Section 3.4.1, $V_0$ be initial wealth and $w$ be a function satisfying
the assumptions of Proposition~\ref{3.3i}.
Our purpose is to find a self-financing trading strategy that satisfies $V_t \geq w(\overline{V}_t)$.
If the value process $V$ of a portfolio satisfies the $w$-drawdown constraint in the sense of Definition~\ref{3.3h},
this condition actually holds.
Then, it is enough to construct a trading strategy whose value process
is equal to an Az{\'e}ma-Yor path $M^U(S)$ with initial value $V_0$.
Recall that an Az{\'e}ma-Yor path associated to the path $S$ and $U \colon [S(0),\infty\mathclose{[} \to \mathbb{R}$ is defined by
$M^U(S) = U(\overline{S}) - U'(\overline{S})(\overline{S} - S)$.
Here we assume $U(S_0) = V_0$.
By Proposition~\ref{3.3b}, this path has another expression
\begin{equation*}
M^U_t(S) = V_0 + \int_0^t U'(\overline{S}_s) dS_s.
\end{equation*}
If the Az{\'e}ma-Yor path $M^U(S)$ is equal to the value process $V$ of some trading strategy, $V$ should satisfy
\begin{equation*}
V_t = V_0 + \int_0^t U'(\overline{S}) dS_s.
\end{equation*}
This is the self-financing condition for a constant riskless asset.
In other words, this is the self-financing condition when we consider the discounted price of the risky asset. 
This equivalence is a well-known result of mathematical finance in the classical It{\^o} calculus framework.
It is also true in our framework.

\begin{Prop} \label{6.2b}
Let $S \in QV(\Pi)$ be a risky asset, $B \in FV_{\mathrm{loc}}$ be a riskless asset with initial price $B_0 = 1$
and $V_0 \geq 0$ be an initial wealth.
Suppose all of $S,S_{-},B,B_{-}$ are strictly positive valued. Define $\widetilde{V} = V/B$ and $\widetilde{S} = S/B$.
Then for a trading strategy $(\xi,\eta)$, the following two conditions are equivalent.
\begin{enumerate}
\item The trading strategy $(\xi,\eta)$ is self-financing.
\item For all $t \in \mathbb{R}_{\geq 0}$,
\begin{equation} \label{6.2ba}
\widetilde{V}_t = V_0 + \int_0^t \xi_{s-} d\widetilde{S}_s.
\end{equation}
\end{enumerate}
\end{Prop}

\begin{proof}
We first show that an admissible integrand $\xi$ is also an admissible integrand of $\widetilde{S}$.
Fix $T > 0$ and choose $A \in FV_{\mathrm{loc}}^m$ and a function $F$ of $C^{1,2}$-class that satisfies
$\xi_t = \frac{d}{dx} F(A_t,S_t)$ ($t \in [0,T]$).
Then $\xi$ can be represented as $\xi = \frac{d}{dx} g(A_t,B_t,\widetilde{S}_t)$,
where $g(a,b,x) = F(a,bx)/b$.
This means that $\xi$ is an admissible integrand of $\widetilde{S}$.

Next we prove that (i) implies (ii).
By Corollary~\ref{2.3s}, we have
\begin{align}
& \frac{V_t}{B_t}
= \frac{V_0}{B_0} + \int_0^t \xi_{s-} \left( S_{s-} d\left(\frac{1}{B}\right)_s + \frac{1}{B_{s-}} dS_s + d\left[ S,\frac{1}{B} \right]_s \right)  \notag \\
& \qquad\quad + \int_0^t \eta_{s-} \left( B_{s-} d\left( \frac{1}{B}\right)_s + \frac{1}{B_{s-}} dB_s + d\left[ B, \frac{1}{B} \right]_s \right),    \notag \\
& \widetilde{S}_t = \frac{S_0}{B_0} + \int_0^t S_{s-} d\left(\frac{1}{B} \right)_s + \int_0^t \frac{1}{B_{s-}} dS_s + \left[ S,\frac{1}{B} \right]_t,  \label{6.2c}  \\
& \int_0^t B_{s-} d\left(\frac{1}{B} \right)_s + \int_0^t \frac{1}{B_{s-}} dB_s + \left[ B, \frac{1}{B} \right]_t = 0.   \label{6.2d}
\end{align}
Combining these three equations, we obtain \eqref{6.2ba}.

It remains to prove that (ii) implies (i).
Using condition (ii), Corollary~\ref{2.3s}, Theorem~\ref{2.3m}, \eqref{6.2c}
and \eqref{6.2d}, we get
\begin{align*} 
V_t
& = V_0 + \int_0^t \xi_{s-} dS_s + \int_0^t \eta_{s-} dB_s + \int_0^t \xi_{s-} \left( B_{s-} d\left[ S,\frac{1}{B} \right]_s + \frac{1}{B_{s-}} d[S,B]_s + d\left[ \left[ S,\frac{1}{B} \right],B \right]_s \right).
\end{align*}
The integrator of the last term satisfies
\begin{align*}
& \int_0^t B_{s-} d\left[ S,\frac{1}{B} \right]_s + \int_0^t \frac{1}{B_{s-}} d[S,B]_s + \left[ \left[ S,\frac{1}{B} \right],B \right]_t  \notag  \\
& = \sum_{0 < s \leq t} \Delta S_s \left( B_{s-}  \Delta \frac{1}{B_s} + \frac{1}{B_{s-}} \Delta B_s +  \Delta \frac{1}{B_{s}} \Delta B_s \right) = 0.
\end{align*}
Therefore, $V$ satisfies the self-financing condition.
\end{proof}

Let us now formulate rigorously the rough observation made before Proposition~\ref{6.2b}.
To apply the results about Az{\'e}ma-Yor paths and drawdown equations,
the running maximum $\overline{S}$ of the risky price process must be continuous.
Under this condition,
the Az{\'e}ma-Yor path $M^U(S)$ associated with $S$ and $U$ satisfies the drawdown constraint 
\begin{equation*}
M^U_t(S) = V_0 + \int_0^t U'(\overline{S}_s) dS_s.
\end{equation*}
Hence, the value process is $V := M^U(S)$ and the trading strategy $\xi := U'(\overline{S})$ is self-financing.
Moreover, if $U$ is that of Proposition~\ref{3.3i}, 
the value process $V$ satisfies the $w$-drawdown constraint.
These observations are summarized as the following proposition.

\begin{Prop} \label{6.2e}
Suppose that the running maximum of the risky price process $S \in QV(\Pi)$ is continuous 
and that the riskless price process is identically 1.
Let $V_0 \geq 0$ be an initial wealth,
and suppose $w \colon [V_0,\infty\mathclose{[} \to \mathbb{R}$ satisfies the same assumptions 
as in Proposition~\ref{3.3i}.
Functions $U$ and $W$ are defined as
\begin{equation*}
W(y) = S_0 \exp\left(\int_{[V_0,y]} \frac{1}{s-w(s)} ds \right), \quad y \in \mathopen{[}a^*,\infty \mathclose{[},
\end{equation*}
and $U = W^{-1}$.
If we set
\begin{equation*}
\xi_t = U'(\overline{S}_t), \quad \eta_t = M^U(S)_t - \xi_t S_t,
\end{equation*}
the trading strategy $(\xi,\eta)$ is self-financing.
Furthermore, the value of portfolio $V := \xi S + \eta$
satisfies the $w$-drawdown constraint in the sense of Definition~\ref{3.3h}.
\end{Prop}

\begin{proof}
The assertion follows from Proposition~\ref{3.3i} immediately.
\end{proof}

\appendix
\section{On the convergence of a sequence of discrete measures on $\mathbb{R}_{\geq 0}$}

Let us consider a sequence of measures $(\mu_n)$ of the form
\begin{equation} \label{a1a}
\mu_n = \sum_{i \geq 0} a^{n}_i \delta_{t^n_i},
\end{equation}
where $a^n_i \geq 0$ and $\pi_n = (t^n_i)_{i \in \mathbb{N}}$ is a partition of $\mathbb{R}_{\geq 0}$.

\begin{Lem} \label{a1b}
Let $\Pi = (\pi_n)_{n \in \mathbb{N}}$ be a partition of $\mathbb{R}_{\geq 0}$ satisfying $\lvert \pi_n \rvert \to 0$.
Suppose that, for each $n \in \mathbb{N}$, $\mu_n$ is a locally finite measure of the form \eqref{a1a} and 
the distribution functions of $(\mu_n)$ converge pointwise to that of a positive locally finite measure $\mu$ with $\mu(\{ 0 \}) = 0$.
Moreover, we assume the following condition.
\begin{itemize}
\item[($*$)] For each $t \in \mathbb{R}_{\geq 0}$ take a sequence $(i_n)$ such that $t \in \mathopen{]}t^n_{i_n},t^n_{i_n+1}]$.
Then $(a^n_{i_n})_{n}$ converges to $\mu(\{ t \})$.
\end{itemize}
Then for all $f \in D(\mathbb{R}_{\geq 0}, \mathbb{R})$ and all $t \in \mathbb{R}_{\geq 0}$, we have
\begin{equation} \label{1e}
\lim_{n \to \infty} \int_{[0,t]} f(s) \mu_n(ds) = \int_{[0,t]} f(s-) \mu(ds).
\end{equation}
\end{Lem}

\begin{Rem} \label{1h}
Condition ($*$) is true if $(\mu_n)$ is given by
\begin{equation*}
\mu_n = \sum_{i} \mu(\mathopen{]}t^n_{i},t^n_{i+1} ]) \delta_{t^n_i}.
\end{equation*}
This condition is also satisfied if $\mu_n = \mu^{\pi_n}_X$ for $X \in QV(\Pi)$.
In this case, Condition ($*$) corresponds to Condition (iii) of Definition~\ref{2.1a}.
\end{Rem}

\begin{proof}
Let $D_t = \left\{ s \in [0,t] \mid \Delta f(s) \neq 0 \right\}$ and $D_t^n = \{ s \in D_t \mid \lvert \Delta f(s) \rvert > 1/n \}$.
For $m \in \mathbb{N}_{\geq 1}$ define finite positive measures on $[0,t]$ by
\begin{gather*}
\mu_{m}'(E) = \mu (D^{m}_t \cap E), \quad \mu_{m}''(E) = \mu (E) - \mu_m'(E) = \mu(E \setminus D^{m}_t),   \notag \\
\mu_{m,n}' = \sum_{t^n_i \leq t } a^n_i 1_{\{ i \mid D^{m}_t \cap \left] t^n_i,t^n_{i+1} \right] \neq \emptyset \}}(i) \delta_{t^n_i}, \quad
\mu_{m,n}'' = \mu_{n} - \mu_{m,n}'.
\end{gather*}
We can deduce from the finiteness of $D^{m}_t$ and the condition ($*$) that
\begin{equation} \label{1f}
\lim_{n \to \infty} \int_{[0,t]} f(s) \mu_{m,n}'(ds) = \int_{[0,t]} f(s-) \mu_{m}'(ds).
\end{equation}
Moreover, we see the distribution functions of $(\mu_{m,n}')_{n}$ converge pointwise to that of $\mu_m'$ on $[0,t]$;
similarly the distribution functions of $(\mu_{m,n}'')_{n}$ converge to that of $\mu_m''$.

Let $C = \sup_{n} \mu_n([0,t]) + \mu([0,t])$.
Now we fix an arbitrary $\delta > 0$ and pick $m$ such that $1/m < \delta/3C$.
Take a function $g$ of the form
\begin{equation*}
g(t) = f(0) 1_{\{ 0 \}} + \sum_{1 \leq i \leq N} b_{i-1} 1_{]s_{i-1},s_{i}]}, \quad 0 = s_0 < s_1 < \dots < s_N \leq t,
\end{equation*}
such that $\lvert g(t) - f(t-) \rvert \leq \delta/3C$ on $[0,t]$.
Then, we have
\begin{align*}
& \left\lvert \int_{[0,t]} f(s) \mu_{m,n}''(ds) - \int_{[0,t]} f(s-) \mu_{m}''(ds)  \right\rvert   \notag \\
& \leq \left\lvert \int_{[0,t]} \left\{ f(s) - f(s-) \right\} \mu_{m,n}''(ds)  \right\rvert + \left\lvert \int_{[0,t]} \left\{ f(s-) - g(s) \right\} \mu_{m,n}''(ds)  \right\rvert  \notag \\
& \quad + \left\lvert \int_{[0,t]} g(s) \mu_{m,n}''(ds) - \int_{[0,t]} g(s) \mu_{m}''(ds) \right\rvert + \left\lvert \int_{[0,t]} \left\{ g(s) - f(s-) \right\} \mu_{m}''(ds) \right\rvert  \notag \\
& = I_1 + I_2 + I_3 + I_4.
\end{align*}
By assumption we see that, for sufficiently large $n$,
\begin{equation*}
I_1 \leq \frac{1}{m} \sup_{n}\mu_n([0,t]) \leq \frac{\delta}{3}, \quad
I_2 \leq \frac{\delta}{3C} \sup_{n} \mu_n([0,t]) \leq \frac{\delta}{3}, \quad 
I_3 \leq \frac{\delta}{3C} \mu([0,t]) \leq \frac{\delta}{3}.
\end{equation*}
Furthermore we have $\lim_{n \to \infty} I_3 = 0$
by assumption. Consequently, 
\begin{equation} \label{1g}
\varlimsup_{n \to \infty} \left\lvert \int_{[0,t]} f(s) \mu_{m,n}''(ds) - \int_{[0,t]} f(s-) \mu_{m}''(ds)  \right\rvert \leq \delta.
\end{equation}
Combining \eqref{1f} and \eqref{1g}, we obtain
\begin{equation*}
\varlimsup_{n \to \infty} \left\lvert \int_{[0,t]} f(s) \mu_n(ds) - \int_{[0,t]} f(s-) \mu(ds)  \right\rvert \leq \delta.
\end{equation*}
Because $\delta$ is chosen arbitrarily, we have the desired result.
\end{proof}

\nocite{Sondermann_2006}
\nocite{Daivs_Obloj_Raval_2014}
\nocite{Foellmer_Schied_2013}
\nocite{Hirai_2016}

\vspace{0.5\baselineskip}

\begin{ackn}
The author would like to thank his supervisor Professor Jun Sekine for helpful support and encouragement.
\end{ackn}

\printbibliography

\vspace{0.5\baselineskip}

Graduate School of Engineering Science\par
Osaka University\par
1-3 Machikaneyama-cho, Toyonaka\par
Osaka 560-8531\par
Japan\par
E-mail: hirai@sigmath.es.osaka-u.ac.jp

\end{document}